\documentclass{amsart}
\usepackage{amsfonts,amssymb,amsmath,amsthm}
\usepackage{url}
\usepackage{enumerate}

\makeatletter
\@namedef{subjclassname@2010}{  \textup{2010} Mathematics Subject Classification.}
\makeatother
\newtheorem{thrm}{Theorem}[section]

\theoremstyle{definition}
\newtheorem{definition}[thrm]{Definition}
\newtheorem{remark}[thrm]{Remark}
\numberwithin{equation}{section}
\email{hzoubeir2014@gmail.com}
\input{tcilatex}

\begin{document}
\title[On polyanalytic functions of Gevrey type on the unit disk]{On the
representation and the uniform polynomial approximation of polyanalytic
functions of Gevrey type on the unit disk}
\dedicatory{\emph{This modest work is dedicated to the memories of two great
men }:\emph{\ our beloved master Ahmed Intissar (1951-2017), a brilliant
mathematician (PhD at M.I.T, Cambridge), a distinguished professor, a man
with a golden heart and our brother and indeed friend Mohamed Saber Bensaid
(1965-2019) the man who belongs to the time of jasmine and sincere love, the
comrade who devoted his whole life to the fight for socialism, democracy,
and human rights.}}
\author{Hicham Zoubeir}
\author{{Samir Kabbaj}}
\address{Universit\'{e} Ibn Tofail, D\'{e}partement de Math\'{e}matiques,
Facult\'{e} des sciences, Kenitra, Maroc.}
\date{}

\begin{abstract}
In this paper we define Gevrey polyanalytic classes of order $N$ on the unit
disk $D$ and we obtain for these classes a characteristic expansion into $N-$%
analytic polynomials on suitable neighborhoods of $D$. As an application of
our main theorem, we perform for the Gevrey polyanalytic classes of order $N$
on the unit disk $D$, an analogue to E. M. Dyn'kin's theorem. We also
derive, for these classes, their characteristic degree of the best uniform
approximation on $D$ by $N-$analytic polynomials.
\end{abstract}

\maketitle

\section{Introduction}

Among the various classes of functions, the class of bianaytic functions,
the class of polyanalytic functions and the Gevrey classes occupy a major
place in mathematical analysis and in mathematical physics. Given a nonempty
open subset $U$ of $%
\mathbb{R}
^{2},$ a function $f:U\rightarrow 
\mathbb{C}
$ is said to be bianalytic if it satisfies, for each $z\in U,$ the condition 
$\left( \frac{\partial }{\partial \overline{\overline{z}}}\right)
^{2}F\left( z\right) =0.$ Bianalytic functions originates from mechanics
where they played a fundamental role in solving the problems of the planar
elasticity theory. Their usefulness in mechanics was illustrated by the
pioneering works of Kolossoff, Muskhelishvili and their followers (\cite%
{KOL1}-\cite{KOL3}, \cite{MUS1}, \cite{MUS2}, \cite{TEO}).\ By the
systematic use of complex variable techniques these authors have greatly
simplified and extended the solutions of the problems of the elasticity
theory. The class of polyanalytic functions of order $N$ $\left( N\in 
\mathbb{N}
^{\ast }\right) $, is a generalisation of the class of analytic functions
and of that of bianalytic functions. The class $H_{N}(U)$ of polyanalytic
functions of order $N$ on $U$ is the set of functions $F:U\rightarrow 
\mathbb{C}
$ of class $C^{N}$ on $U$ such that the following condition holds for each $%
z\in U:\left( \frac{\partial }{\partial \overline{z}}\right) ^{N}F(z)=0$.
The class of polyanalytic functions was studied intensively by the russian
school under the supervision of M. B. Balk (\cite{BAL}). The lines of
current research on polyanalytic functions are various : the problem of the
best uniform approximation by $N-$analytic polynomials (\cite{MAZ1}-\cite%
{MAZ3}, \cite{VER}), the study of wavelets and Gabor frames (\cite{ABR2}-%
\cite{ABR5}, \cite{ABR9}), the time-frequency analysis (\cite{ABR5}, \cite%
{ABR7}, \cite{ABR8}), the sampling and interpolation in function spaces (%
\cite{ABR1}), the study of coherent states in quantum mechanics (\cite{HAI}, 
\cite{MOU1}), \cite{MOU2}), the image and signal processing (\cite{ABR6}, 
\cite{ABR7}), etc. Gevrey classes, which are also, but in a completely
different way, a generalisation of real analytic functions, were first
introduced by Gevrey $\left( \text{\cite{GEV}}\right) $. Indeed they are
intermediate spaces between the space of real-analytic functions and the
space of $C^{\infty }$ smooth functions. Given $s>0,$ the Gevrey class $%
G^{s}\left( U\right) $ is defined as the set of all functions $%
f:U\rightarrow 
\mathbb{C}
,$ of class\ $C^{\infty }$ on \ $U$ such that\ there exist a constant $R>0$
satisfying for every $x\in U$ and $\left( \alpha _{1},\alpha _{2}\right) \in 
\mathbb{N}
^{2}$ the following estimate : $\left \vert \frac{\partial ^{\alpha
_{1}+\alpha _{2}}g}{\partial x_{1}^{\alpha _{1}}\partial x_{2}^{\alpha _{2}}}%
(x)\right \vert \leq R^{\alpha _{1}+\alpha _{2}+1}\left( \alpha _{1}+\alpha
_{2}\right) ^{s\left( \alpha _{1}+\alpha _{2}\right) }$. The Gevrey classes
play an important role in various branches of partial and ordinary
differential equations and especially in the analysis of operators whose
properties cannot be apprehended by the classical analytic framework. \ The
field of applications of Gevrey classes is very wide : the Gevrey class
regularity\emph{\ }of\emph{\ }the equations of mathematical physics (\cite%
{KIM}, \cite{KUK}, \cite{LAR}, \cite{TIT1}, \cite{TIT2}, \cite{XU}),\emph{\ }%
the study of singularities in micro-local analysis (\cite{ROD1}-\cite{ROD3}%
), the Gevrey solvability of differential operators (\cite{CIC}, \cite{ROD2}%
, \cite{ROD3}), the divergent series and singular and the singular
differential equations (\cite{MIY}, \cite{RAM1}), the study of dynamical
systems (\cite{GRAM}, \cite{RAM2}), the evolution partial differential
equations (\cite{FER}, \cite{LEV}, \cite{TEM}), etc. However despite this
great interest devoted to polyanalytic functions and to Gevrey classes there
is still, at our knowledge, no interplay between them. Our main goal in this
paper is then to contribute to bridging this gap. In order to achieve this
task, we consider the intersection of a Gevrey class on the unit disk $D$
and the class of polyanalytic functions of order $N$ on $D$ and then look
for some property which characterizes this class of functions. Indeed we
obtain, in the main result of this paper, the complete description of these
so-called Gevrey polyanalytic classes of order $N$ by specific expansions
into $N-$analytic polynomials on suitable neighborhoods of $D$. We establish
two applications of our main theorem. The first application concerns the
proof for the Gevrey polyanalytic classes of order $N$ on the unit disk $D$
of an analogue to the E. M. Dyn'kin's theorem (\cite{DYN}). Let us recall
that this theorem basically says that a function $f:\mathcal{R}\rightarrow 
\mathbb{C}
$ of class $C^{\infty }$on a region $\mathcal{R}$ of $%
\mathbb{C}
$ belongs to a class $\mathfrak{X}$ of smooth complex valued functions on $%
\mathcal{R}$ if and only if it has an extension $F:$ $%
\mathbb{C}
\rightarrow 
\mathbb{C}
$ of class $C^{1}$ on $%
\mathbb{C}
$ so that $\frac{\partial F}{\partial \overline{z}}$ satisfies a growth
condition of the form $\left \vert \frac{\partial F}{\partial \overline{z}}%
\left( z\right) \right \vert \leq A\mathcal{H}_{\mathfrak{X}}\left( B\varrho
\left( z,\mathcal{R}\right) \right) $, $z\in 
\mathbb{C}
\backslash \mathcal{R},$ where $A,B>0$ are constants, $\varrho \left( z,%
\mathcal{R}\right) $ is the euclidean distance from $z$ to $\mathcal{R}$ and 
$\mathcal{H}_{\mathfrak{X}}:%
\mathbb{R}
^{+\ast }\rightarrow 
\mathbb{R}
$ is a function related to the class $\mathfrak{X}$, depending only on this
class and such that $\underset{t\rightarrow 0,t>0}{\lim }\mathcal{H}_{%
\mathfrak{X}}\left( t\right) =0.$ $\mathcal{H}_{\mathfrak{X}}$ is then
called the weight function of the class $\mathfrak{X}$ while the function $F$
is said to be a pseudoanalytic extension of the function $f$ with respect to
the class $\mathfrak{X}.$ The second application of our main theorem
concerns the construction, for Gevrey polyanalytic classes of order $N$, of
their degree of the best uniform approximation on $D$ by $N-$analytic
polynomials.

The paper is structured as follows. In section $2$, we state some notations
and definitions and prove a fundamental result which is necessary for the
proof of the first application of our main result. In section $3,$ we give
the definition of polyanalytic functions of order $N$ and recall their main
properties. In section $4,$ we recall the definition of Gevrey classes and
the quantitative version of the closure of these classes under the
composition of functions and finally we state the definition of polyanalytic
Gevrey classes on an open subset $U$ of the complex plane$.$ In section $5,$
we state the main result of this paper. The section $6$ is devoted to the
proof of the main result. Section $7$ presents our applications of the main
result of the paper$.$ Finally section $8$ is an appendix which provides the
proofs of some technical estimates which are crucial for the proof of three
results : the proof of the direct part for $N=1$ $\left( \text{proposition }%
8.1.\right) ,$ the proof of the converse part of corollary $1$ (proposition $%
8.2.),$ the proof of the converse part of corollary $2$ (proposition $8.3.).$

\section{Preliminary notes}

\subsection{Basic notations}

Let $h$ a function defined on a nonempty subset $E$ of $%
\mathbb{C}
.$ We denote by $||h||_{\infty ,E}$ the quantity :%
\begin{equation*}
||h||_{\infty ,E}:=\underset{u\in E}{\sup }|h(u)|\in 
\mathbb{R}
^{+}\cup \left \{ +\infty \right \}
\end{equation*}

Let $S$ be a nonempty subset of $%
\mathbb{C}
,$ then we set for each $z\in 
\mathbb{C}
:$%
\begin{equation*}
\varrho \left( z,S\right) :=\underset{u\in S}{\inf }\left \vert z-u\right
\vert
\end{equation*}%
$\varrho \left( z,S\right) $ represents the euclidean distance from $z$ to $%
S $.

For all $x\in 
\mathbb{R}
$ we set :%
\begin{equation*}
\left \lfloor x\right \rfloor :=\max \left( \{p\in 
\mathbb{Z}
:p\leq x\} \right)
\end{equation*}

We set for every $n\in 
\mathbb{N}
:$%
\begin{equation*}
J\left( n\right) :=\left \{ p\in 
\mathbb{N}
:0\leq p\leq n-1\right \}
\end{equation*}%
We set for each $\alpha \in 
\mathbb{N}
^{n}$ and $s\in I\left( n\right) :$%
\begin{equation*}
\alpha !:=\overset{n}{\underset{j=1}{\dprod }}\alpha _{j}!,\text{ }\left
\vert \alpha \right \vert :=\overset{n}{\underset{j=1}{\dsum }}\alpha _{j}
\end{equation*}

Let $\sigma :=\left( \sigma _{1},\sigma _{2}\right) ,$ $\sigma ^{\prime
}:=\left( \sigma _{1}^{\prime },\sigma _{2}^{\prime }\right) \in 
\mathbb{N}
^{2}.$ We set :%
\begin{equation*}
\sigma \preccurlyeq \sigma ^{\prime }\Leftrightarrow (\sigma _{1}\leq \sigma
_{1}^{\prime }\text{ and }\sigma _{2}\leq \sigma _{2}^{\prime })
\end{equation*}%
In this case we set :%
\begin{equation*}
\binom{\sigma ^{\prime }}{\sigma }:=\binom{\sigma _{1}^{\prime }}{\sigma _{1}%
}\binom{\sigma _{2}^{\prime }}{\sigma _{2}}
\end{equation*}

We denote by $d\nu \left( \zeta \right) $ the usual Lebesgue measure on $%
\mathbb{C}
.$

Let $\zeta \in 
\mathbb{C}
$ and $r>0.$ We denote by $\Delta \left( \zeta ,r\right) $ $\left( \text{%
resp. }\overline{\Delta }\left( \zeta ,r\right) \right) $ the usual open $%
\left( \text{resp. closed}\right) $ disk of center $\zeta $ and radius $r.$ $%
\Gamma (\zeta ,r)$ denotes the usual circle of center $\zeta $ and radius $%
r. $ $D:=$ $\Delta \left( 0,1\right) $ (resp. $\overline{D}$ $:=\overline{%
\Delta }\left( 0,1\right) $) is called the open (resp.closed) unit disk of
the complex plane. It is then clear that :%
\begin{equation*}
\varrho \left( z,D\right) =\left \{ 
\begin{array}{c}
0\text{ if }z\in D \\ 
\left \vert z\right \vert -1\text{ else}%
\end{array}%
\right.
\end{equation*}%
For all $r,R>0$ and $n\in 
\mathbb{N}
^{\ast }$ we set :%
\begin{equation*}
D_{R}:=\Delta \left( 0,1+R\right) ,\text{ }\overline{D}_{R}:=\overline{%
\Delta }\left( 0,1+R\right) ,\text{ }D_{k,R,n}:=D_{Rn^{\frac{-1}{k}}}
\end{equation*}

For each $m\in 
\mathbb{N}
^{\ast }\backslash \left \{ 1\right \} $ we denote by $\mathcal{L}_{m}$ the
function defined on the set $\mathcal{U}_{m}:=\left \{ \left(
s_{1},...,s_{m-1}\right) \in 
\mathbb{R}
^{m-1}:1<s_{1}<...<s_{m-1}\right \} $ by the formula :%
\begin{eqnarray*}
&&\mathcal{L}_{m}\left( s_{1},...,s_{m-1}\right) \\
&:&=\underset{j=1}{\overset{m-1}{\dprod }}\left( \frac{s_{j}^{2}+1}{%
s_{j}^{2}-1}\right) \underset{p=1}{\overset{m-1}{\sum }}\left( s_{p}^{m-1}%
\underset{j\neq p}{\dprod }\left( \frac{s_{j}^{2}+1}{\left \vert
s_{j}^{2}-s_{p}^{2}\right \vert }\right) \right)
\end{eqnarray*}%
For each real numbers $r$ and $r_{0}$ such that $r>$ $r_{0}>0$ we set :%
\begin{eqnarray*}
&&\mathcal{J}_{m}\left( r_{0},r\right) \\
&:&=\mathcal{L}_{m}\left( 1+\frac{r-r_{0}}{mr_{0}},1+2\left( \frac{r-r_{0}}{%
mr_{0}}\right) ...,1+\left( m-1\right) \left( \frac{r-r_{0}}{mr_{0}}\right)
\right)
\end{eqnarray*}

\begin{proposition}
The following estimate holds for each $m\in 
\mathbb{N}
^{\ast }$ and $\varepsilon >0:$%
\begin{equation}
\mathcal{J}_{m}\left( 1+\frac{\varepsilon }{2},1+\varepsilon \right) \leq
\left( m-1\right) \left( \frac{5m}{2}\right) ^{2m-2}\left( 1+\frac{2}{%
\varepsilon }\right) ^{2m-2}  \label{ESTM.I}
\end{equation}
\end{proposition}

\begin{proof}
Indeed we have : 
\begin{eqnarray*}
&&\mathcal{J}_{m}\left( 1+\frac{\varepsilon }{2},1+\varepsilon \right) \\
&=&\mathcal{L}_{m}\left( 1+\frac{1}{m}\left( \frac{\varepsilon }{%
2+\varepsilon }\right) ,1+\frac{2}{m}\left( \frac{\varepsilon }{%
2+\varepsilon }\right) ,...,1+\left( \frac{m-1}{m}\right) \left( \frac{%
\varepsilon }{2+\varepsilon }\right) \right) \\
&=&\underset{j=1}{\overset{m-1}{\dprod }}\left( \frac{\left( 1+\frac{j}{m}%
\left( \frac{\varepsilon }{2+\varepsilon }\right) \right) ^{2}+1}{\left( 1+%
\frac{j}{m}\left( \frac{\varepsilon }{2+\varepsilon }\right) \right) ^{2}-1}%
\right) \underset{p=1}{\overset{m-1}{\sum }}\left( 
\begin{array}{c}
\left( 1+\frac{p}{m}\left( \frac{\varepsilon }{2+\varepsilon }\right)
\right) ^{m-1}\cdot \\ 
\cdot \underset{j\neq p}{\dprod }\left( \frac{\left( 1+\frac{j}{m}\left( 
\frac{\varepsilon }{2+\varepsilon }\right) \right) ^{2}+1}{\left \vert
\left( 1+\frac{j}{m}\left( \frac{\varepsilon }{2+\varepsilon }\right)
\right) ^{2}-\left( 1+\frac{p}{m}\left( \frac{\varepsilon }{2+\varepsilon }%
\right) \right) ^{2}\right \vert }\right)%
\end{array}%
\right) \\
&\leq &\underset{j=1}{\overset{m-1}{\dprod }}\left( \frac{5m\left( \frac{%
2+\varepsilon }{\varepsilon }\right) }{j\left( 2+\frac{j}{m}\left( \frac{%
\varepsilon }{2+\varepsilon }\right) \right) }\right) \underset{p=1}{\overset%
{m-1}{\sum }}\left( 2^{m-1}\underset{j\neq p}{\dprod }\left( \frac{5m\left( 
\frac{2+\varepsilon }{\varepsilon }\right) }{\left \vert j-p\right \vert
\left( 2+\frac{\left( j+p\right) }{m}\left( \frac{\varepsilon }{%
2+\varepsilon }\right) \right) }\right) \right) \\
&\leq &\left( m-1\right) \left( \frac{5m}{2}\right) ^{2m-2}\left( 1+\frac{2}{%
\varepsilon }\right) ^{2m-2}
\end{eqnarray*}

Thence we achieve the proof of the proposition.
\end{proof}

From now on $N\in 
\mathbb{N}
^{\ast }$ and $k>0$ are arbitrary but fixed real numbers.

\subsection{Some function spaces and differential operators}

$C\left( \overline{D}\right) $ represents the set of continuous complex
valued functions on $\overline{D}$ while $C_{0}^{\infty }\left( 
\mathbb{C}
\right) $ denotes the set of complex valued functions defined and of class $%
C^{\infty }$ on $%
\mathbb{C}
$ and of compact support.

We denote by $\frac{\partial }{\partial \overline{z}}$ the well-known
Cauchy-Riemann operator differential operator defined by the formula : 
\begin{equation*}
\frac{\partial }{\partial \overline{z}}:=\frac{1}{2}\left( \frac{\partial }{%
\partial x}+i\frac{\partial }{\partial y}\right)
\end{equation*}%
while $\frac{\partial }{\partial z}$ is the differential operator whose
definition is :%
\begin{equation*}
\frac{\partial }{\partial z}:=\frac{1}{2}\left( \frac{\partial }{\partial x}%
-i\frac{\partial }{\partial y}\right)
\end{equation*}%
For each $\alpha :=\left( \alpha _{1},\alpha _{2}\right) \in 
\mathbb{N}
^{2},$ we denote by $D^{\alpha }$ the differential operator :

\begin{equation*}
D^{\alpha }:=\frac{\partial ^{\left \vert \alpha \right \vert }}{\partial
x^{\alpha _{1}}\partial y^{\alpha _{2}}}
\end{equation*}
while $\frac{\partial ^{\left \vert \alpha \right \vert }}{\partial
z^{\alpha _{1}}\partial \overline{z}^{\alpha _{2}}}$ is the differential
operator defined by :%
\begin{equation*}
\frac{\partial ^{\left \vert \alpha \right \vert }}{\partial z^{\alpha
_{1}}\partial \overline{z}^{\alpha _{2}}}:=\left( \frac{\partial }{\partial z%
}\right) ^{\alpha _{1}}\left( \frac{\partial }{\partial \overline{z}}\right)
^{\alpha _{2}}
\end{equation*}

The following proposition will play a fundamental role in the proof of the
first application of our main result.

\begin{proposition}
\textit{For every }$\varphi \in C_{0}^{\infty }\left( 
\mathbb{C}
\right) $ \textit{and }$N\in 
\mathbb{N}
^{\ast }$ \textit{the following relation holds }:%
\begin{equation*}
\varphi \left( z\right) =\underset{%
\mathbb{C}
}{\diint }\frac{\left( \overline{z}-\overline{\zeta }\right) ^{N-1}}{\pi
\left( N-1\right) !\left( z-\zeta \right) }\left( \frac{\partial }{\partial 
\overline{z}}\right) ^{N}\varphi \left( \zeta \right) d\nu \left( \zeta
\right)
\end{equation*}
\end{proposition}

\begin{proof}
It is well-known ( \cite{VLA}; page 126, exercise 11-4) that the function : 
\begin{equation*}
\mathcal{V}_{N}:\left( x,y\right) \mapsto \frac{\left( x-iy\right) ^{N-1}}{%
\pi .\left( N-1\right) !\left( x+iy\right) }
\end{equation*}%
is a fundamental solution of the differential operator $\left( \frac{%
\partial }{\partial \overline{z}}\right) ^{N}.$ Let us then write $\mathcal{V%
}_{N}\left( x,y\right) $ in the form :%
\begin{equation*}
\mathcal{V}_{N}\left( z\right) :=\frac{\overline{z}^{N-1}}{\pi .\left(
N-1\right) !z}
\end{equation*}%
where $z:=x+iy.$ Consequently for each $\varphi \in C_{0}^{\infty }\left( 
\mathbb{C}
\right) $ and $N\in 
\mathbb{N}
^{\ast }$ the following formula holds for all $z\in 
\mathbb{C}
:$%
\begin{eqnarray*}
\varphi \left( z\right) &=&\underset{%
\mathbb{C}
}{\diint }\mathcal{V}_{N}\left( z-\zeta \right) \left( \frac{\partial }{%
\partial \overline{z}}\right) ^{N}\varphi \left( \zeta \right) d\nu \left(
\zeta \right) \\
&=&\underset{%
\mathbb{C}
}{\diint }\frac{\left( \overline{z}-\overline{\zeta }\right) ^{N-1}}{\pi
\left( N-1\right) !\left( z-\zeta \right) }\left( \frac{\partial }{\partial 
\overline{z}}\right) ^{N}\varphi \left( \zeta \right) d\nu \left( \zeta
\right)
\end{eqnarray*}

The proof of the proposition is then complete.
\end{proof}

\section{Polyanalytic functions of order $N:$ definition and main properties}

Let $U$ be a nonempty open subset of $%
\mathbb{C}
.$ The set $H_{N}(U)$ of polyanalytic functions of order $N$ on $U$ is the
set of functions $F:U\rightarrow 
\mathbb{C}
$ of class $C^{N}$ on $U$ such that :%
\begin{equation*}
\left( \forall z\in U\right) :\left( \frac{\partial }{\partial \overline{z}}%
\right) ^{N}F(z)=0
\end{equation*}%
Then $H_{1}(U)$ is the set of holomorphic functions on $U,$ while $H_{2}(U)$
is the set of bianalytic functions on $U.$ It is well known (\cite{BAL},
pages 10 and 11) that a function $F:U\rightarrow 
\mathbb{C}
$ is polyanalytic of order $N$ if and only if it is of the form :

\begin{equation}
\left( \forall z\in U\right) :F(z)=\underset{p=0}{\overset{N-1}{\sum }}%
F_{p}(z)\overline{z}^{p}  \label{representation}
\end{equation}%
where $F_{0},$ $...,F_{N-1}$ are holomorphic on $U.$ We can prove by an easy
induction on $N$ that the representation (\ref{representation}) of $F$ is
unique. For every $p\in J\left( N\right) ,$ the function $F_{p}$ is called
the holomorphic component of order $p$ of $F$ and labelled by the notation $%
F_{p}:=\mathcal{K}_{p}\left( F\right) .$ It follows also from the formula (%
\ref{representation}) that every function $f\in H_{N}\left( U\right) $ is of
class $C^{\infty }$ on $U.$ We denote by $\Pi _{N}$ the vector space of
complex polynomial functions $P$ of the form : 
\begin{equation*}
z\in 
\mathbb{C}
\mapsto P(z):=\underset{p=0}{\overset{N-1}{\sum }}Q_{p}(z)\overline{z}^{p},
\end{equation*}%
where $Q_{0},,...,Q_{N-1}$ are holomorphic polynomials. The members of $\Pi
_{N}$ are called $N$-analytic polynomials. The degree of $P\neq 0$ is then
the integer :%
\begin{equation*}
d%
{{}^\circ}%
\left( P\right) :=\underset{0\leq p\leq N-1,\text{ }Q_{p}\neq 0}{\max }\deg
\left( Q_{p}\right)
\end{equation*}%
where $\deg \left( Q_{p}\right) $ denotes the usual degree of $Q_{p}\in 
\mathbb{C}
\left[ X\right] .$ We will set $d%
{{}^\circ}%
\left( 0\right) =-\infty .$ With the convention that%
\begin{equation*}
\left( \forall n\in 
\mathbb{N}
\right) :-\infty \leq n
\end{equation*}%
$\Pi _{N,n}$ is, for each $n\in 
\mathbb{N}
,$ the vector subspace of$\  \Pi _{N}$ defined by :%
\begin{equation*}
\Pi _{N,n}:=\left \{ P\in \Pi _{N}:d%
{{}^\circ}%
\left( P\right) \leq n\right \}
\end{equation*}%
For each continuous function $f:\overline{D}\rightarrow 
\mathbb{C}
,$ the $N-$approximating number of order $n$ is :%
\begin{equation*}
\mathcal{E}_{N,n}\left( f\right) :=\underset{P\in \Pi _{N,n}}{\inf }\left
\Vert f-P\right \Vert _{\infty ,\overline{D}}
\end{equation*}%
Let $\mathfrak{X}$ be a nonempty subset of $C\left( \overline{D}\right) $. A
degree of the best uniform $N-$polynomial approximation of functions of $%
\mathfrak{X}$ is a set $\mathcal{R}:\mathcal{=}\left \{ \mathfrak{m}_{\mu
}:\mu \in \Lambda \right \} $ of functions $\mathfrak{m}_{\mu }:%
\mathbb{R}
^{+}\rightarrow 
\mathbb{R}
^{+}$ ($\Lambda $ a nonempty set) such that :%
\begin{equation*}
\left( \forall f\in C\left( \overline{D}\right) \right) :\left[ \left( f\in 
\mathfrak{X}\right) \Leftrightarrow \left( \left( \exists \mu \in \Lambda
\right) \left( \forall n\in 
\mathbb{N}
\right) :\mathcal{E}_{N,n}\left( f\right) \leq \mathfrak{m}_{\mu }\left(
n\right) \right) \right]
\end{equation*}

The following results, whose proofs can be found in (\cite{BAL}, pages 21-25
and 27), will play a fundamental role in the current paper.

\begin{theorem}
\textit{Let }$\gamma _{p}:=\Gamma (\zeta ,r_{p})$ \textit{be }$N$ \textit{%
circles where }$0<r_{0}<r_{1}<...<r_{N-1}<r$ and $f$ $\in H_{N}\left( \Delta
(\zeta ,r)\right) .$\textit{\ }
\end{theorem}

1.\emph{\ Maximum modulus principle for polyanalytic functions of order }$N$%
\emph{\ }$\emph{on}$\emph{\ }$\Delta (\zeta ,r)$\emph{\ }$:$\textit{\ }

a) \textit{The following estimates hold}%
\begin{equation}
\left \Vert f\right \Vert _{\infty ,\overline{\Delta }(\zeta ,r_{0})}\leq 
\mathcal{L}_{N}\left( \frac{r_{1}}{r_{0}},...,\frac{r_{N-1}}{r_{0}}\right) 
\underset{0\leq p\leq N-1}{\max }\left \Vert f\right \Vert _{\infty ,\gamma
_{p}}  \label{maxp0}
\end{equation}

b) \textit{If we assume that :}%
\begin{equation*}
\left \Vert f\right \Vert _{\infty ,\Delta (\zeta ,r)\diagdown \overline{%
\Delta }(\zeta ,r_{0})}<+\infty
\end{equation*}%
\textit{\ then the following estimates hold }:%
\begin{equation}
\left \Vert f\right \Vert _{\infty ,\overline{\Delta }(\zeta ,r_{0})}\leq 
\mathcal{J}_{N}\left( r_{0},r\right) \left \Vert f\right \Vert _{\infty
,\Delta (\zeta ,r)\diagdown \overline{\Delta }(\zeta ,r_{0})}  \label{maxp1}
\end{equation}%
\begin{equation}
\left( \forall p\in J\left( N\right) \right) :\left \Vert \mathcal{K}%
_{p}\left( f\right) \right \Vert _{\infty ,\overline{\Delta }(\zeta
,r_{0})}\leq \frac{\mathcal{J}_{N}\left( r_{0},r\right) \left \Vert f\right
\Vert _{\infty ,\Delta (\zeta ,r)\diagdown \overline{\Delta }(\zeta ,r_{0})}%
}{r_{0}^{p}}  \label{maxp2}
\end{equation}

2. \emph{Weierstrass theorem for polyanalytic functions of order }$\emph{on}$%
\emph{\ }$\Delta (\zeta ,r)$\emph{\ }$:$

\textit{Let }$\left( f_{n}\right) _{n\geq 1}$ \textit{be} \textit{a sequence}
\textit{of polyanalytic}\emph{\ }\textit{functions of order}\emph{\ }$N$%
\emph{\ on }$\Delta \left( \zeta ,r\right) $\textit{\ which is uniformly
convergent on every compact subset of }$\Delta \left( \zeta ,r\right) $ 
\textit{to a function} $f$ . \textit{Then }$f\in H_{N}\left( \Delta \left(
\zeta ,r\right) \right) $ \textit{and for each }$\left( p,q\right) \in 
\mathbb{N}
^{2}$\textit{\ the sequence }$\left( \frac{\partial ^{p+q}f_{n}}{\partial
z^{p}\partial \overline{z}^{q}}\right) _{n\geq 1}$\textit{\ is uniformly
convergent on every compact subset of }$\Delta \left( \zeta ,r\right) $ 
\textit{to the function} $\frac{\partial ^{p+q}f}{\partial z^{p}\partial 
\overline{z}^{q}}$.

By means of the theorem $3.1.$ we prove easily the following result.

\begin{proposition}
\textit{Let }$\left( f_{n}\right) _{n\geq 1}$ \textit{be} \textit{a sequence}
\textit{of }poly\textit{analytic functions of order }$N$\textit{\ on an open
disk }$\Delta \left( \zeta ,r\right) .$ \textit{Let us assume that the
sequence }$\left( f_{n}\right) _{n\geq 1}$ \textit{is uniformly convergent
on every compact subset of }$\Delta \left( \zeta ,r\right) $ \textit{to the
function} $f$ $.$ \textit{Then }$f\in H_{N}\left( \Delta \left( \zeta
,r\right) \right) $ \textit{and, for every }$p\in J\left( N-1\right) ,$%
\textit{\ the sequence of functions }$\left( \mathcal{K}_{p}\left(
f_{n}\right) \right) _{n\geq 1}$ \textit{is uniformly convergent on every
compact subset of }$\Delta \left( \zeta ,r\right) $ \textit{to the function }%
$\mathcal{K}_{p}\left( f\right) .$
\end{proposition}

The following result plays a crucial role in the proof \ of the second
application of the main result of this paper.

\begin{proposition}
\emph{Bernstein-Walsh inequality for N-analytic polynomials on the unit disk}
$:$
\end{proposition}

\textit{For each }$P\in \Pi _{N,n}$ and $z\in 
\mathbb{C}
\setminus D,$ \textit{the following inequality holds }:%
\begin{equation*}
\left \vert P\left( z\right) \right \vert \leq \left( 2^{N+1}-1\right) 
\mathcal{J}_{N}\left( \frac{1}{2},1\right) \left \Vert P\right \Vert
_{\infty ,\overline{D}}\left \vert z\right \vert ^{n+N-1}
\end{equation*}

\begin{proof}
We have for each $j\in J\left( N\right) :$%
\begin{equation*}
\left \Vert \mathcal{K}_{j}\left( P\right) \right \Vert _{\infty ,\overline{%
\Delta }_{\frac{1}{2}}}\leq 2^{j}\mathcal{J}_{N}\left( \frac{1}{2},1\right)
\left \Vert P\right \Vert _{\infty ,\overline{D}}
\end{equation*}%
On the other hand, by virtue of the well known Bernstein-Walsh inequality (%
\cite{BAG}), we have for every $z\in $ $%
\mathbb{C}
\setminus D$ and $j\in J\left( N-1\right) :$%
\begin{equation*}
\left \vert \mathcal{K}_{j}\left( P\right) \left( z\right) \right \vert \leq
2^{n}\left \Vert \mathcal{K}_{j}\left( P\right) \right \Vert _{\infty ,%
\overline{\Delta }_{\frac{1}{2}}}\left \vert z\right \vert ^{n}
\end{equation*}%
It follows that :%
\begin{eqnarray*}
\left \vert P\left( z\right) \right \vert &\leq &\underset{j=0}{\overset{N-1}%
{\sum }}\left \vert \mathcal{K}_{j}\left( P\right) \left( z\right) \right
\vert \left \vert z\right \vert ^{j} \\
&\leq &\underset{j=0}{\overset{N-1}{\sum }}2^{j}\mathcal{J}_{N}\left( \frac{1%
}{2},1\right) \left \Vert P\right \Vert _{\infty ,\overline{D}}\left \vert
z\right \vert ^{n+j} \\
&\leq &\left( 2^{N+1}-1\right) \mathcal{J}_{N}\left( \frac{1}{2},1\right)
\left \Vert P\right \Vert _{\infty ,\overline{D}}\left \vert z\right \vert
^{n+N-1}
\end{eqnarray*}

Thence we achieve the proof of the proposition.
\end{proof}

\section{Gevrey classes and Gevrey polyanalytic classes of order $N$}

\begin{definition}
Let $k,s>0$ and $N\in 
\mathbb{N}
^{\ast }$ be given fixed numbers. \textit{Let }$U$ be a nonempty subset of $%
\mathbb{C}
$ and $I$ an interval of $%
\mathbb{R}
.$
\end{definition}

\begin{enumerate}
\item \textit{The\ Gevrey class }$G^{s}\left( U\right) $ \textit{is the set
of functions }$f:U\rightarrow 
\mathbb{C}
$\textit{\ of class }$C^{\infty }$\textit{\ on }$U$\textit{\ such that }$:$%
\begin{equation*}
\left( \forall \alpha \in 
\mathbb{N}
^{2}\right) :\left \Vert D^{\alpha }f\right \Vert _{\infty ,U}\leq
B_{0}^{\left \vert \alpha \right \vert +1}\left \vert \alpha \right \vert
^{s\left \vert \alpha \right \vert }
\end{equation*}%
$B_{0}>0$\textit{\ being a constant, with the convention that }$0^{0}=1.$

\item \textit{The\ Gevrey class }$G^{s}\left( I\right) $ \textit{is the set
of functions }$f:I\rightarrow 
\mathbb{C}
$\textit{\ of class }$C^{\infty }$\textit{\ on }$I$\textit{\ such that }$:$%
\begin{equation*}
\left( \forall n\in 
\mathbb{N}
\right) :\left \Vert f^{\left( n\right) }\right \Vert _{\infty ,I}\leq
B_{1}^{n+1}n^{sn}
\end{equation*}%
$B_{1}>0$\textit{\ being a constant}$.$

\item \textit{The\ Gevrey polyanalytic class of order }$N$ \textit{on }$U$%
\textit{, }$H_{N}^{k}(U),$ \textit{is the set of functions }$f\in H_{N}(U)$%
\textit{\ such that }$:$%
\begin{equation*}
\left( \forall \alpha \in 
\mathbb{N}
^{2}\right) :\left \Vert D^{\alpha }f\right \Vert _{\infty ,U}\leq
B_{2}^{\left \vert \alpha \right \vert +1}\left \vert \alpha \right \vert
^{\left( 1+\frac{1}{k}\right) \left \vert \alpha \right \vert }
\end{equation*}%
$B_{2}>0$\textit{\ being a constant}$.$ \textit{It follows that} :%
\begin{equation*}
H_{N}^{k}(U)=H_{N}(U)\cap G^{1+\frac{1}{k}}\left( U\right)
\end{equation*}
\end{enumerate}

\begin{remark}
We prove easily by direct computations that $H_{N}^{k}(U)$\ is the set of
functions $f\in H_{N}(U)$\ such that $:$%
\begin{equation*}
\left( \forall \left( n,m\right) \in 
\mathbb{N}
^{2}\right) :\left \Vert \frac{\partial ^{n+m}f}{\partial z^{n}\partial 
\overline{z}^{m}}\right \Vert _{\infty ,U}\leq B_{3}^{n+m+1}\left(
n+m\right) ^{\left( 1+\frac{1}{k}\right) \left( n+m\right) }
\end{equation*}%
$B_{3}>0$ being a real constant. \textit{\ }
\end{remark}

\begin{remark}
\textit{We prove, by an easy induction on }$N\geq 1,\ $\textit{that the
following equivalence holds for each }$f\in H_{N}\left( U\right) :$\textit{\ 
}%
\begin{equation*}
\mathit{\ }f\in H_{N}^{k}\mathit{\ }\left( U\right) \Leftrightarrow \left[
\left( \forall p\in J\left( N\right) \right) :\mathcal{K}_{p}\left( f\right)
\in H_{1}^{k}\left( U\right) \right]
\end{equation*}
\end{remark}

A slight refinement of the proof by D. Figueirinhas (\cite{FIG}, theorem $2.$
$5.$ pages $11-13)$ of the closure of Gevrey classes under composition
provide us with the following result which is essential for the proof of the
direct part of our main result.

\begin{theorem}
Let $I$ be an interval of $%
\mathbb{R}
$ and $f\in G^{s}\left( I\right) $, $s>1$. Let $U$ be an open set of $%
\mathbb{C}
$ and $g:U\rightarrow 
\mathbb{C}
$ a function of class $C^{\infty }$ on $J$ such that $f\left( I\right)
\subset U$ and $g\in G^{s}\left( U\right) $. The function $h:=g\circ f\ $\
belongs also to$\ G^{s}\left( I\right) $ and if we assume that $:$%
\begin{eqnarray*}
\left \Vert f^{\left( n\right) }\right \Vert _{\infty ,I} &\leq
&c_{1}d_{1}^{n}n^{sn},\text{ }n\in 
\mathbb{N}
\\
\left \Vert D^{\alpha }g\right \Vert _{\infty ,U} &\leq &c_{2}d_{2}^{\left
\vert \alpha \right \vert }\left \vert \alpha \right \vert ^{s\left \vert
\alpha \right \vert },\text{ }\alpha \in 
\mathbb{N}
^{2}
\end{eqnarray*}%
$\left( c_{1},\text{ }d_{1},\text{ }c_{2},\text{ }d_{2}>0\text{ being
constants}\right) $ then we will have $:$%
\begin{equation*}
\left \Vert h^{\left( n\right) }\right \Vert _{\infty ,I}\leq
c_{3}d_{3}^{n}n^{sn},\text{ }n\in 
\mathbb{N}%
\end{equation*}%
where $:$%
\begin{equation*}
c_{3}=\frac{ec_{1}c_{2}d_{2}}{1+ec_{1}d_{2}},\text{ }d_{3}=d_{1}\left(
1+ec_{1}d_{2}\right)
\end{equation*}%
\textit{\ }
\end{theorem}

\section{Statement of the main result}

Our main result, in this paper, is the following.

\begin{enumerate}
\item \textit{Let }$F\in H_{N}^{k}(D).$ \textit{Then there exist constants }$%
C>0,$\textit{\ }$R>0$\textit{\ and }$\delta \in ]0,1[$\textit{\ and a
sequence }$(P_{n})_{n\in 
\mathbb{N}
^{\ast }}$\textit{\ of }$N-$\textit{analytic polynomials\ such that }$:$%
\begin{equation*}
\left \{ 
\begin{array}{c}
\left( \forall n\in 
\mathbb{N}
^{\ast }\right) :||P_{n}||_{\infty ,D_{k,R,n}}\leq C\delta ^{n} \\ 
\left( \forall z\in D\right) :\overset{+\infty }{\underset{n=1}{\sum }}%
P_{n}(z)=F(z) \\ 
\left( \forall n\in 
\mathbb{N}
^{\ast }\right) :d%
{{}^\circ}%
\left( P_{n}\right) \leq n^{\frac{k+1}{k}}%
\end{array}%
\right.
\end{equation*}%
\textit{\  \  \ }

\item \textit{Conversly, let }$(f_{n})_{n\in 
\mathbb{N}
^{\ast }}$\textit{be a sequence of }$N-$\textit{analytic polynomials\ such
that }$:$%
\begin{equation*}
\left( \forall n\in 
\mathbb{N}
^{\ast }\right) :||f_{n}||_{\infty ,D_{k,R,n}}\leq C\delta ^{n}
\end{equation*}%
\textit{for some constants }$C>0,$\textit{\ }$R>0$\textit{\ and }$\delta \in
]0,1[.$\textit{\ Then the function series }$\sum f_{n}$\textit{\ converges
uniformly on }$D$\textit{\ to a function }$f\in H_{N}^{k}(D).$
\end{enumerate}

\section{Proof of the main result}

\subsection{Proof of the direct part of the main result}

\subsubsection{Proof of the direct part for N=1}

\begin{lemma}
\bigskip \textit{Let}%
\begin{equation*}
\mathit{\ }%
\begin{array}{cccc}
F: & D & \rightarrow & 
\mathbb{C}
\\ 
& z & \mapsto & \underset{p=0}{\overset{+\infty }{\sum }}a_{p}z^{p}%
\end{array}%
\end{equation*}%
\textit{\ be a function which belongs to }$H_{1}^{k}(D).$\textit{\ Then
there exists constants }$\mathcal{P}_{1},$\textit{\ }$\mathcal{P}_{2}>0$%
\textit{\ such that }$:$%
\begin{equation*}
\left( \forall p\in 
\mathbb{N}
\right) :|a_{p}|\leq \mathcal{P}_{1}\exp \left( -\mathcal{P}_{2}p^{\frac{k}{%
k+1}}\right)
\end{equation*}
\end{lemma}

\begin{proof}
Let us associate to every $t\in \lbrack 0,1[$ the function : 
\begin{equation*}
\mathit{\ }%
\begin{array}{cccc}
\varphi _{t}: & [0,2\pi ] & \rightarrow & 
\mathbb{C}
\\ 
& \theta & \mapsto & F\left( te^{i\theta }\right)%
\end{array}%
\end{equation*}%
Assume that $F\in H_{1}^{k}(D),$ then there exists $\mathcal{P}_{0}\geq 1$
such that :%
\begin{equation*}
\left( \forall n\in 
\mathbb{N}
\right) :||F^{(n)}||_{\infty ,D}\leq \mathcal{P}_{0}^{n+1}n^{\left( 1+\frac{1%
}{k}\right) n}
\end{equation*}%
Direct computations based on the result of theorem 4. 4. prove that :%
\begin{eqnarray*}
\left( \forall t\in \lbrack 0,1[\right) \left( \forall n\in 
\mathbb{N}
^{\ast }\right) &:& \\
||\varphi _{t}^{(n)}||_{\infty ,[0,2\pi ]} &\leq &\left( 1+e\mathcal{P}%
_{0}\right) ^{n}n^{n(1+\frac{1}{k})}
\end{eqnarray*}%
On the other hand for every $n\in 
\mathbb{N}
^{\ast }$ and $t\in \left] 0,1\right[ ,$ the function $\varphi _{t}^{(n)}$
is a $2\pi -$periodic function whose Fourier coefficient of order $p\in 
\mathbb{N}
^{\ast }$ is :%
\begin{equation*}
i^{n}p^{n}t^{p}a_{p}=\frac{1}{2\pi }\underset{0}{\overset{2\pi }{\int }}%
\varphi _{t}^{(n)}(\theta )e^{-ip\theta }d\theta
\end{equation*}%
It follows that :%
\begin{eqnarray*}
\left( \text{ }\forall \left( p,n\right) \in \left( 
\mathbb{N}
^{\ast }\right) ^{2}\right) \left( \forall \text{ }t\in \left] 0,1\right[
\right) &:& \\
|a_{p}| &\leq &\frac{1}{p^{n}t^{p}}(1+e\mathcal{P}_{0})^{n}n^{n(1+\frac{1}{k}%
)}
\end{eqnarray*}%
Thence we have :%
\begin{equation*}
\left( \forall p\in 
\mathbb{N}
^{\ast }\right) :|a_{p}|\leq 2\mathcal{P}_{0}\underset{n\in 
\mathbb{N}
^{\ast }}{\inf }n^{n(1+\frac{1}{k})}\left( \frac{1+e\mathcal{P}_{0}}{p}%
\right) ^{n}
\end{equation*}%
But straightforward computations show that the following relation holds for
all $p\in $\ $%
\mathbb{N}
^{\ast }$:%
\begin{eqnarray*}
&&\underset{n\in 
\mathbb{N}
^{\ast }}{\inf }n^{n(1+\frac{1}{k})}\left( \frac{1+e\mathcal{P}_{0}}{p}%
\right) ^{n} \\
&=&\min \left( \left( \left( \frac{1+e\mathcal{P}_{0}}{p}\right) ^{\frac{k}{%
k+1}}r_{p}\right) ^{(1+\frac{1}{k})r_{p}},\left( \left( \frac{1+e\mathcal{P}%
_{0}}{p}\right) ^{\frac{k}{k+1}}\left( r_{p}+1\right) \right) ^{(1+\frac{1}{k%
})\left( r_{p}+1\right) }\right)
\end{eqnarray*}%
where $r_{p}:=\left \lfloor e^{-1}\left( 1+e\mathcal{P}_{0}\right) ^{-\frac{k%
}{k+1}}p^{\frac{k}{k+1}}\right \rfloor $. Consequently there exists two
constants \textit{\ }$\mathcal{P}_{1},$\textit{\ }$\mathcal{P}_{2}>0$\textit{%
\ } such that :%
\begin{equation*}
\left( \forall p\in 
\mathbb{N}
^{\ast }\right) :\underset{n\in 
\mathbb{N}
^{\ast }}{\inf }\frac{n^{n(1+\frac{1}{k})}}{\left( \frac{p}{1+e\mathcal{P}%
_{0}}\right) ^{n}}\leq \frac{\mathcal{P}_{1}}{1+e\mathcal{P}_{0}}\exp \left(
-\mathcal{P}_{2}p^{\frac{k}{k+1}}\right)
\end{equation*}%
It follows that : 
\begin{equation*}
\left( \forall p\in 
\mathbb{N}
^{\ast }\right) :|a_{p}|\leq \mathcal{P}_{1}\exp \left( -\mathcal{P}_{2}p^{%
\frac{k}{k+1}}\right)
\end{equation*}

Thence we achieve the proof of the proposition.
\end{proof}

End of the proof of the direct part for $N=1.$

Let us set for every $z\in 
\mathbb{C}
$ and $n\in 
\mathbb{N}
:$%
\begin{equation*}
P_{n}\left( z\right) :=\underset{2^{-\left( \frac{k+1}{k}\right) }n^{\frac{%
k+1}{k}}\leq p<2^{-\left( \frac{k+1}{k}\right) }\left( n+1\right) ^{\frac{k+1%
}{k}}}{\sum }a_{p}z^{p}
\end{equation*}%
Then $(P_{n})_{n\in 
\mathbb{N}
^{\ast }}$\textit{\ }is a sequence of $1-$analytic polynomials such that :%
\begin{equation*}
\left( \forall n\in 
\mathbb{N}
^{\ast }\right) :d^{\circ }\left( P_{n}\right) \leq n^{\frac{k+1}{k}}
\end{equation*}%
Furthermore the following inequality holds for each $n\in 
\mathbb{N}
^{\ast }$ and $z\in D_{k,\frac{\mathcal{P}_{2}}{4},n}:$%
\begin{eqnarray*}
&&\left \vert P_{n}\left( z\right) \right \vert \\
&\leq &\underset{2^{-\left( \frac{k+1}{k}\right) }n^{\frac{k+1}{k}}\leq
p<2^{-\left( \frac{k+1}{k}\right) }\left( n+1\right) ^{\frac{k+1}{k}}}{\sum }%
\left \vert a_{p}\right \vert \left \vert z\right \vert ^{p} \\
&\leq &\underset{2^{-\left( \frac{k+1}{k}\right) }n^{\frac{k+1}{k}}\leq
p<2^{-\left( \frac{k+1}{k}\right) }\left( n+1\right) ^{\frac{k+1}{k}}}{\sum }%
\mathcal{P}_{1}\exp \left( -\mathcal{P}_{2}p^{\frac{k}{k+1}}\right) \cdot \\
&&\cdot \left( 1+\frac{\mathcal{P}_{2}}{4}n^{-\frac{1}{k}}\right) ^{p} \\
&\leq &\mathcal{P}_{1}\exp \left( -\frac{\mathcal{P}_{2}}{2}n\right) \left(
2^{-\left( \frac{k+1}{k}\right) }\left( n+1\right) ^{\frac{k+1}{k}%
}-2^{-\left( \frac{k+1}{k}\right) }n^{\frac{k+1}{k}}+1\right) \cdot \\
&&\cdot \left( 1+\frac{\mathcal{P}_{2}}{4}n^{-\frac{1}{k}}\right) ^{n^{\frac{%
k+1}{k}}} \\
&\leq &2^{\frac{k+1}{k}}\mathcal{P}_{1}\exp \left( -\frac{\mathcal{P}_{2}}{2}%
n\right) \left( \left( n+1\right) ^{\frac{k+1}{k}}-n^{\frac{k+1}{k}%
}+1\right) \left( 1+\frac{\mathcal{P}_{2}}{4}n^{-\frac{1}{k}}\right) ^{n^{%
\frac{k+1}{k}}}
\end{eqnarray*}%
But we will prove in the proposition $8.1.$ in the appendix below, that
there exists a constant $\mathcal{P}_{3}>0$ such that%
\begin{eqnarray*}
&&\left( \forall n\in 
\mathbb{N}
\right) :2^{\frac{k+1}{k}}\mathcal{P}_{1}\exp \left( -\frac{\mathcal{P}_{2}}{%
2}n\right) \left( \left( n+1\right) ^{\frac{k+1}{k}}-n^{\frac{k+1}{k}%
}+1\right) \cdot \\
&&\cdot \left( 1+\frac{\mathcal{P}_{2}}{4}n^{-\frac{1}{k}}\right) ^{n^{\frac{%
k+1}{k}}} \\
&\leq &\mathcal{P}_{3}\left( e^{-\frac{\mathcal{P}_{2}}{8}}\right) ^{n}
\end{eqnarray*}

Consequently the following relation holds :%
\begin{equation*}
\left( \forall n\in 
\mathbb{N}
\right) :\left \Vert P_{n}\right \Vert _{\infty ,D_{k,\frac{\mathcal{P}_{2}}{%
4},n}}\leq \mathcal{P}_{3}\left( e^{-\frac{\mathcal{P}_{2}}{8}}\right) ^{n}
\end{equation*}%
But $e^{-\frac{\mathcal{P}_{2}}{8}}\in \left] 0,1\right[ $ thence we achieve
the proof of the direct part of the main result for $N=1.$

$\square $

\subsubsection{Proof of the direct part in the general case $N\geq 2$}

Let $f\in H_{N}^{k}(D).$ Then there exist, thanks to the remark 3 above, a
functions $g_{0},...,g_{N}$ belonging to $H_{1}^{k}\left( D\right) $ such
that :%
\begin{equation*}
\left( \forall z\in D\right) :f\left( z\right) =\underset{p=0}{\overset{N-1}{%
\sum }}g_{p}\left( z\right) \overline{z}^{p}
\end{equation*}%
Hence there exist for each $p\in J\left( N\right) $ a triple $\left(
C_{p},R_{p},\delta _{p}\right) \in 
\mathbb{R}
^{+\ast }\times 
\mathbb{R}
^{+\ast }\times \left] 0,1\right[ $ of constants and a sequence $\left(
Q_{n,p}\right) _{n\in 
\mathbb{N}
^{\ast }}$ of holomorphic polynomials such that :%
\begin{equation*}
\left \{ 
\begin{array}{c}
\left( \forall n\in 
\mathbb{N}
^{\ast }\right) :d^{\circ }\left( Q_{n,p}\right) \leq n^{\frac{k+1}{k}} \\ 
\left( \forall n\in 
\mathbb{N}
\right) :\left \Vert Q_{n,p}\right \Vert _{\infty ,D_{k,R_{p},n}}\leq
C_{p}\delta _{p}^{n} \\ 
\left( \forall z\in D\right) :\overset{+\infty }{\underset{n=1}{\sum }}%
Q_{n,p}(z)=g_{p}(z)%
\end{array}%
\right.
\end{equation*}%
Let us set $R:=\min \left( R_{0},...,R_{N-1}\right) ,$ $\delta :=\max \left(
\delta _{0},...,\delta _{N-1}\right) ,$ $C:=\max \left(
C_{0},...,C_{N-1}\right) ,$ $Q_{n}\left( z\right) :=\overset{N-1}{\underset{%
p=0}{\sum }}Q_{n,p}\left( z\right) \overline{z}^{p}$. Then $\delta \in \left]
0,1\right[ $ and $Q_{n\text{ }}$ is for each $n\in 
\mathbb{N}
$ a $N-$analytic polynomial such that :%
\begin{equation*}
\left \{ 
\begin{array}{c}
\left( \forall n\in 
\mathbb{N}
^{\ast }\right) :d^{\circ }\left( Q_{n}\right) \leq n^{\frac{k+1}{k}} \\ 
\left( \forall n\in 
\mathbb{N}
\right) :\left \Vert Q_{n}\right \Vert _{\infty ,D_{k,R,n}}\leq NC\left(
1+R\right) ^{N-1}\delta ^{n}%
\end{array}%
\right.
\end{equation*}%
Furthermore we have for each $z\in D:$%
\begin{eqnarray*}
f\left( z\right) &=&\underset{p=0}{\overset{N}{\sum }}g_{p}\left( z\right) 
\overline{z}^{p} \\
&=&\underset{n=0}{\overset{+\infty }{\sum }}Q_{n}\left( z\right)
\end{eqnarray*}%
Thence we achieve the proof of the direct part of the main result.

$\square $

\subsection{Proof of the converse part of the main result}

Let $A>0$ and for each $n\in 
\mathbb{N}
,$ a function $f_{n}\in H_{1}(D_{k,A,n})$ such that :%
\begin{equation*}
\left \{ 
\begin{array}{c}
\left( \forall n\in 
\mathbb{N}
^{\ast }\right) :f_{n}\in H_{N}(D_{k,A,n}) \\ 
\left( \forall n\in 
\mathbb{N}
^{\ast }\right) :||f_{n}||_{\infty ,D_{k,A,n}}\leq C\rho ^{n}%
\end{array}%
\right. \text{ }
\end{equation*}%
where $C>0,$ $\rho \in \left] 0,1\right[ $ are constant. Without loss of the
generality we can also assume that $A<1.$ It follows, by virtue of the
Weierstrass theorem for polyanalytic functions of order $N$, that the
function series $\sum {f_{n}|}_{D}$ converges uniformly to a function $f$ $%
\in H_{N}(D).$ For every $n\in 
\mathbb{N}
^{\ast }$ and $p\in J\left( N\right) ,$ let $a_{p,n}$ be the holomorphic
component of order $p$ of $f_{n}.$ Let $a_{p}$ be the holomorphic component
of order $p$ of $f.$ By virtue of the remark, for each $p\in J\left(
N\right) $ the function series $\sum a_{p,n}$ is uniformly convergent on
every compact of $D$ to the function $a_{p}.$ Then the inequalities (\ref%
{maxp2}) and (\ref{ESTM.I})\ entail that the following estimate holds for\
each\ $\left( p,n\right) \in 
\mathbb{N}
^{2}$ : 
\begin{eqnarray*}
\left \Vert a_{p,n}\right \Vert _{\infty ,D_{k,\frac{A}{2},n}} &\leq &%
\mathcal{J}_{N}\left( 1+\frac{A}{2}n^{-\frac{1}{k}},1+An^{-\frac{1}{k}%
}\right) C\rho ^{n} \\
&\leq &C\left( N-1\right) \left( \frac{5}{2}N\right) ^{2N-2}\left( 1+\frac{2%
}{A}n^{\frac{1}{k}}\right) ^{2N-2}\rho ^{n} \\
&\leq &C\left( N-1\right) \left( \frac{10}{A}N\right) ^{2N-2}n^{\frac{2N-2}{k%
}}\rho ^{n} \\
&\leq &C\left( N-1\right) \left( \frac{10}{A}N\right) ^{2N-2}\left( \underset%
{t\geq 0}{\sup }t^{\frac{2N-2}{k}}\sqrt{\rho }^{t}\right) \sqrt{\rho }^{n}
\end{eqnarray*}%
But direct computations prove that :%
\begin{equation}
\left( \forall l\in \mathbb{N}\right) :\underset{t\geq 0}{\sup }\left( t^{%
\frac{l}{k}}\sqrt{\rho }^{t}\right) =\left( \left( \frac{2}{ek\ln \left( 
\frac{1}{\rho }\right) }\right) ^{\frac{1}{k}}\right) ^{l}l^{\frac{l}{k}}
\label{boundsup}
\end{equation}%
It follows that :%
\begin{equation*}
\left \Vert a_{p,n}\right \Vert _{\infty ,D_{k,\frac{A}{2},n}}\leq \mathcal{Q%
}\sqrt{\rho }^{n}
\end{equation*}%
where :%
\begin{equation*}
\mathcal{Q}:=C\left( N-1\right) \left( \frac{10}{A}N\right) ^{2N-2}\left(
2N-2\right) ^{\frac{2N-2}{k}}\left( \left( \frac{2}{ek\ln \left( \frac{1}{%
\rho }\right) }\right) ^{\frac{1}{k}}\right) ^{2N-2}
\end{equation*}%
Thence, in view of the Cauchy's inequalities for holomorphic functions, we
can write :

\begin{eqnarray*}
\left( \forall n\in 
\mathbb{N}
^{\ast }\right) &:&\left \Vert a_{p,n}^{\left( l\right) }\mathit{\ }\right
\Vert _{\infty ,D} \\
&\leq &\mathcal{Q}l!\left( \frac{2}{A}\right) ^{l}n^{\frac{l}{k}}e^{-\ln
\left( \frac{1}{\sqrt[4]{\rho }}\right) n}\sqrt[4]{\rho }^{n} \\
&\leq &\mathcal{Q}\left( \frac{2}{A}\left( \frac{1}{ek\ln \left( \frac{1}{%
\rho }\right) }\right) ^{\frac{1}{k}}\right) ^{l}l^{\left( 1+\frac{1}{k}%
\right) l}\sqrt[4]{\rho }^{n}
\end{eqnarray*}%
Thence the following estimates hold :%
\begin{equation}
\left( \forall \left( l,n\right) \in 
\mathbb{N}
\times 
\mathbb{N}
^{\ast }\right) :\Vert {a_{p,n}^{(l)}}\Vert _{\infty ,D}\leq \mathcal{P}%
_{4}^{l+1}l^{\left( 1+\frac{1}{k}\right) l}\sqrt[4]{\rho }^{n}  \label{estim}
\end{equation}%
where $\mathcal{P}_{4}:=\max \left( \mathcal{Q},\mathcal{Q}\left( \frac{2}{A}%
\left( \frac{1}{ek\ln \left( \frac{1}{\rho }\right) }\right) ^{\frac{1}{k}%
}\right) ^{l},1\right) .$ Consequently the function series $\sum
a_{p,n}^{(l)}$ is uniformly convergent on each compact subset of $D,$ and we
obtain the following estimate :%
\begin{equation*}
\left( \forall l\in \mathbb{N}\right) :||a_{p}^{(l)}||_{\infty ,D}\leq \frac{%
\mathcal{P}_{4}}{1-\sqrt[4]{\rho }}\mathcal{P}_{4}^{l}l^{\left( 1+\frac{1}{k}%
\right) l}
\end{equation*}%
Consequently : 
\begin{equation*}
\left( \forall p\in J\left( N\right) \right) :a_{p}\in H_{1}^{k}(D)
\end{equation*}%
Let us consider the function $f\in H_{N}\left( D\right) $ defined by the
formula :%
\begin{equation*}
\left( \forall z\in D\right) :f(z):=\underset{j=0}{\overset{N-1}{\sum }}%
a_{p}(z)\overline{z}^{p}
\end{equation*}%
Then we have for each $\left( l,m\right) \in \mathbb{N}^{2}:$%
\begin{eqnarray*}
\left( \forall z\in D\right) &:& \\
\left \vert \frac{\partial ^{l+m}f}{\partial z^{l}\partial \overline{z}^{m}}%
(z)\right \vert &\leq &\underset{j=0}{\overset{N-1}{\sum }}p!\left \vert
a_{p}^{\left( l\right) }(z)\right \vert \\
&\leq &\frac{N!N\mathcal{P}_{4}}{1-\sqrt[4]{\rho }}\mathcal{P}%
_{4}^{l}l^{\left( 1+\frac{1}{k}\right) l} \\
&\leq &\frac{N!N\mathcal{P}_{4}}{1-\sqrt[4]{\rho }}\mathcal{P}%
_{4}^{l+m}\left( l+m\right) ^{\left( 1+\frac{1}{k}\right) \left( l+m\right) }
\end{eqnarray*}%
It follows that $f\in H_{N}^{k}\left( D\right) .$

The proof of the converse part of the main result is then achieved.

$\square $

\begin{remark}
\textit{It follows easily from the relation (\ref{boundsup}) that if }$%
\left( f_{n}\right) _{n\in 
\mathbb{N}
^{\ast }}$\textit{\ is a sequence of }$N-$\textit{analytic polynomials\ such
that}%
\begin{equation*}
\left( \forall n\in 
\mathbb{N}
^{\ast }\right) :||f_{n}||_{\infty ,\text{ }D_{k,A,n}}\leq C\rho ^{n}
\end{equation*}%
\textit{for some constants }$C>0,$\textit{\ }$A>0$\textit{\ and }$\rho \in
]0,1[$\textit{\ then there exist some constants }$\mathcal{P}\geq 1$ such
that \textit{the following estimates hold for each }$\left( n,\alpha \right)
\in 
\mathbb{N}
^{\ast }\times 
\mathbb{N}
^{2}$ $:$
\end{remark}

\begin{equation*}
\left \Vert D^{\alpha }f_{n}\mathit{\ }\right \Vert _{\infty ,\text{ }D_{k,%
\frac{A}{3},n}}\leq \mathcal{P}^{\left \vert \alpha \right \vert +1}\left
\vert \alpha \right \vert ^{\left( 1+\frac{1}{k}\right) \left \vert \alpha
\right \vert }\sqrt[4]{\rho }^{n}
\end{equation*}

\section{Applications}

\subsection{E. M. Dyn'kin's theorem for the Gevrey class $H_{N}^{k}\left(
D\right) $}

\begin{corollary}
\textit{Let }$f\in H_{N}\left( D\right) .$\textit{\ }
\end{corollary}

1. \textit{\ If }$f\in H_{N}^{k}\left( D\right) $\textit{\ then there exists
a function }$F$ $:%
\mathbb{C}
\rightarrow $ $%
\mathbb{C}
$ \textit{of class }$C^{\infty }$\textit{\ on }$%
\mathbb{C}
\;$\textit{with compact\ support such that }$:$%
\begin{equation*}
\left \{ 
\begin{array}{c}
F|_{D}=f \\ 
\left( \forall z\in \mathbb{C}\setminus \overline{D}\right) :\left \vert
\left( \frac{\partial }{\partial \overline{z}}\right) ^{N}F(z)\right \vert
\leq C_{1}\exp \left[ -C_{2}\left( \left \vert z\right \vert -1\right) ^{-k}%
\right]%
\end{array}%
\right.
\end{equation*}%
\textit{where }$C_{1},$\textit{\ }$C_{2}$\textit{\ }$>0$\textit{\ are
constants.}

2. \textit{Conversly, if there exists a function }$F$ $:\mathcal{\ U}%
\rightarrow $ $%
\mathbb{C}
$ \textit{of class }$C^{\infty }$\textit{\ on an open neighborhood }$%
\mathcal{U}\;$\textit{of the closed unit disk }$\overline{D}$\textit{\ such
that }$:$%
\begin{equation*}
\left \{ 
\begin{array}{c}
F|_{D}=f \\ 
\left( \forall z\in \mathcal{U}\setminus \overline{D}\right) :\left \vert
\left( \frac{\partial }{\partial \overline{z}}\right) ^{N}F(z)\right \vert
\leq C_{1}\exp \left[ -C_{2}\left( \left \vert z\right \vert -1\right) ^{-k}%
\right]%
\end{array}%
\right.
\end{equation*}%
\textit{where }$C_{1},$\textit{\ }$C_{2}$\textit{\ }$>0$\textit{\ are
constants then }$f\in H_{N}^{k}\left( D\right) .$

\begin{proof}
1. Assume that $f\in H_{N}^{k}\left( D\right) $\ then, according to theorem $%
5.$ $1.$, there exists constants $A\in ]0,1[,$ $C>0,$ $\rho \in ]0,1[$ and a
sequence of $N-$polynomial functions $(f_{n})_{n\in 
\mathbb{N}
}$ such that :%
\begin{equation*}
\left \{ 
\begin{array}{c}
\left( \forall n\in 
\mathbb{N}
\right) :\Vert f_{n}\Vert _{\infty ,D_{k,A,n}}\leq C\rho ^{n} \\ 
\left( \forall x\in D\right) :\sum_{n=0}^{+\infty }f_{n}(x)=f(x)%
\end{array}%
\right.
\end{equation*}%
On the other hand there exist (\cite{TOU}, lemma $3.3.$, page $77$) for each 
$n\in 
\mathbb{N}
$ a function $h_{n}:%
\mathbb{C}
\longrightarrow \lbrack 0,1]$ and a family of positive constants $(L_{\nu
})_{\nu \in 
\mathbb{N}
^{2}}$ such that : 
\begin{equation*}
\left \{ 
\begin{array}{c}
\left( \forall z\in D_{k,\frac{A}{4},n}\right) :h_{n}(z)=1 \\ 
\left( \forall \;z\in 
\mathbb{C}
\setminus D_{k,\frac{A}{3},n}\right) :h_{n}(z)=0 \\ 
\left( \forall \nu \in 
\mathbb{N}
^{2}\right) \left( \forall z\in 
\mathbb{R}
^{2}\right) :|D^{\nu }h_{n}(z)|\leq L_{\nu }n^{\frac{|\nu |}{k}}%
\end{array}%
\right.
\end{equation*}%
We set for each $p\in 
\mathbb{N}
:$%
\begin{equation*}
\mathcal{M}_{p}:=\underset{\left \vert \nu \right \vert \leq p}{\max }L_{\nu
}
\end{equation*}%
We denote by $F_{n}$ the function defined by :%
\begin{equation*}
\left \{ 
\begin{array}{c}
\left( \forall z\in D_{k,A,n}\right) :F_{n}(z)=h_{n}(z)f_{n}(z) \\ 
\left( \forall z\in 
\mathbb{C}
\setminus D_{k,A,n}\right) :F_{n}(z)=0%
\end{array}%
\right.
\end{equation*}%
The function $F_{n}$ is of class $C^{\infty }$ on $%
\mathbb{C}
$ and satisfies the condition :%
\begin{equation*}
{F_{n}|}_{D_{k,\frac{A}{4},n}}={f_{n}|}_{D_{k,\frac{A}{4},n}}
\end{equation*}%
Since :%
\begin{equation*}
\text{ }\left( \forall n\in 
\mathbb{N}
\right) :||F_{n}||_{\infty ,\mathbb{C}}\leq C\rho ^{n}
\end{equation*}%
it follows that the function series $\sum F_{n}$ is uniformly convergent on $%
\mathbb{C}
$ to a continuous function $F$ with compact support contained in $D_{A}.$
Furthermore we have for all $z\in D:$%
\begin{eqnarray*}
F(z) &=&\underset{n=0}{\overset{+\infty }{\sum }}F_{n}(z) \\
&=&\underset{n=0}{\overset{+\infty }{\sum }}f_{n}(z) \\
&=&f(z)
\end{eqnarray*}%
Thence $F$ is an extension to $%
\mathbb{C}
$ of $f.$ Let $n\in 
\mathbb{N}
,$ $\alpha \in 
\mathbb{N}
^{2}$ and $z\in 
\mathbb{C}
.$ If $z\in 
\mathbb{C}
\setminus D_{k,\frac{A}{3},n}$ then we will have : 
\begin{equation*}
D^{\alpha }F_{n}(z)=0
\end{equation*}%
Now if $z\in D_{k,\frac{A}{3},n}$ then we have $:$%
\begin{eqnarray*}
&&|D^{\alpha }F_{n}(z)| \\
&\leq &\sum_{\beta \preccurlyeq \alpha }(_{\beta }^{\alpha })|D^{\beta
}h_{n}(z)|\text{ }|D^{\alpha -\beta }f_{n}(z)|
\end{eqnarray*}%
But there exists, thanks to remark $4.2.$, a constant $\mathcal{P}_{5}\geq 1$
such that the following estimate holds for each $v\in 
\mathbb{N}
^{2}$ :%
\begin{equation*}
\left \Vert D^{v}f_{n}\mathit{\ }\right \Vert _{\infty ,\text{ }D_{k,\frac{A%
}{3},n}}\leq \mathcal{P}_{5}^{\left \vert v\right \vert +1}\left \vert
v\right \vert ^{\left( 1+\frac{1}{k}\right) \left \vert v\right \vert }\sqrt[%
4]{\rho }^{n}
\end{equation*}%
Hence we can write :%
\begin{eqnarray*}
&&|D^{\alpha }F_{n}(z)| \\
&\leq &\sum_{\beta \preccurlyeq \alpha }(_{\beta }^{\alpha })\mathcal{M}%
_{\left \vert \alpha \right \vert }n^{\frac{|\beta |}{k}}|D^{\alpha -\beta
}f_{n}(z)| \\
&\leq &\sum_{\beta \preccurlyeq \alpha }\mathcal{M}_{\left \vert \alpha
\right \vert }(_{\beta }^{\alpha })n^{\frac{|\beta |}{k}}\mathcal{P}%
_{5}^{\left \vert \alpha -\beta \right \vert +1}\left \vert \alpha -\beta
\right \vert ^{\left( 1+\frac{1}{k}\right) \left \vert \alpha -\beta \right
\vert }\sqrt[4]{\rho }^{n} \\
&\leq &\mathcal{P}_{5}\mathcal{M}_{\left \vert \alpha \right \vert }\left( 2%
\mathcal{P}_{5}\right) ^{\left \vert \alpha \right \vert }\left \vert \alpha
\right \vert ^{\left( 1+\frac{1}{k}\right) \left \vert \alpha \right \vert
}n^{\frac{|\alpha |}{k}}\sqrt[4]{\rho }^{n}
\end{eqnarray*}%
It follows that the function series $\sum D^{\alpha }F_{n}(z)$ is for all $%
\alpha \in 
\mathbb{N}
^{2}$ normally convergent on $%
\mathbb{C}
.$ Hence the function $F=\sum_{n=1}^{\infty }F_{n}$ is of class $C^{\infty }$
on $%
\mathbb{C}
$ and we have for $\ $each $z\in 
\mathbb{C}
\setminus D$ and $\alpha \in 
\mathbb{N}
^{2}:$%
\begin{equation*}
\left \{ 
\begin{array}{c}
D^{\alpha }F(z)=\sum_{n=1}^{+\infty }D^{\alpha }F_{n}(z) \\ 
\left( \frac{\partial }{\partial \overline{z}}\right)
^{N}F(z)=\sum_{n=1}^{+\infty }\left( \frac{\partial }{\partial \overline{z}}%
\right) ^{N}F_{n}(z)%
\end{array}%
\right.
\end{equation*}%
On the other hand we have for each $n\in 
\mathbb{N}
^{\ast }:$%
\begin{equation*}
\left( \frac{\partial }{\partial \overline{z}}\right) ^{N}F_{n}(z)=0\text{ if%
}\; \varrho (z,D)\geq \frac{A}{3}n^{-\frac{1}{k}}\; \text{or}\; \; \varrho
(z,D)<\frac{A}{4}n^{-\frac{1}{k}}
\end{equation*}%
But if $\; \frac{A}{4}n^{-\frac{1}{k}}\leq \varrho (z,D)\leq $ $\frac{A}{3}%
n^{-\frac{1}{k}}$ then we have the following estimates : 
\begin{eqnarray*}
\left \vert \left( \frac{\partial }{\partial \overline{z}}\right)
^{N}F_{n}(z)\right \vert &\leq &\underset{p=0}{\overset{N}{\sum }}\frac{%
\binom{N}{p}}{2^{N}}\left \vert D^{\left( p,N-p\right) }F_{n}(z)\right \vert
\\
&\leq &\mathcal{P}_{5}\mathcal{M}_{N}\left( 2\mathcal{P}_{5}\right)
^{N}N^{\left( 1+\frac{1}{k}\right) N}n^{\frac{N}{k}}\sqrt[4]{\rho }^{n} \\
&\leq &\mathcal{P}_{6}e^{\frac{\ln \left( \rho \right) }{8}n}\underset{t\geq
0}{\sup }t^{\frac{N}{k}}e^{-\frac{\ln \left( \frac{1}{\rho }\right) }{8}t}
\end{eqnarray*}%
where $\mathcal{P}_{6}:=\mathcal{P}_{5}\mathcal{M}_{N}\left( 2\mathcal{P}%
_{5}\right) ^{N}N^{\left( 1+\frac{1}{k}\right) N}.$ The relation (\ref%
{boundsup}) entails then that the following relation holds :%
\begin{equation*}
\left \vert \left( \frac{\partial }{\partial \overline{z}}\right)
^{N}F_{n}(z)\right \vert \leq \mathcal{P}_{6}\left( \frac{8N}{ek\ln \left( 
\frac{1}{\rho }\right) }\right) ^{\frac{N}{k}}e^{\frac{\ln \left( \rho
\right) }{8}n}
\end{equation*}%
Let us then set $\mathcal{P}_{7}:=\frac{\mathcal{P}_{6}\left( \frac{8N}{%
ek\ln \left( \frac{1}{\rho }\right) }\right) ^{\frac{N}{k}}}{1-e^{\frac{\ln
\left( \rho \right) }{8}}}$ and $\mathcal{P}_{8}:=\frac{1}{4}\left( \frac{A}{%
4}\right) ^{k}\ln \left( \frac{1}{\rho }\right) .$Thence we have for every $%
z\in 
\mathbb{C}
\setminus \overline{D}$ the following estimate :%
\begin{eqnarray*}
\left \vert \left( \frac{\partial }{\partial \overline{z}}\right)
^{N}F(z)\right \vert &\leq &\left( 1-e^{\frac{\ln \left( \rho \right) }{8}%
}\right) \mathcal{P}_{7}\sum_{\frac{A}{4}n^{-\frac{1}{k}}\leq \varrho
(z,D)\leq \frac{A}{3}n^{-\frac{1}{k}}}e^{\frac{\ln \left( \rho \right) }{4}n}
\\
&\leq &\left( 1-e^{\frac{\ln \left( \rho \right) }{8}}\right) \mathcal{P}%
_{7}\sum_{\left( \frac{A}{4\varrho (z,D)}\right) ^{k}\leq n}e^{\frac{\ln
\left( \rho \right) }{4}n} \\
&\leq &\mathcal{P}_{7}\exp \left( -\mathcal{P}_{8}\varrho (z,D)^{-k}\right)
\\
&\leq &\mathcal{P}_{7}\exp \left( -\mathcal{P}_{8}\left( \left \vert z\right
\vert -1\right) ^{-k}\right)
\end{eqnarray*}%
2. \ Let $f\in H_{N}\left( D\right) .$\ Assume that there exists a function $%
F$ $:\  \mathcal{U}\rightarrow $ $%
\mathbb{C}
$ of class $C^{\infty }$\ on a neighborhood $\mathcal{U}\;$of the closed
unit disk $\overline{D}$ such that $:$%
\begin{equation*}
\left \{ 
\begin{array}{c}
F|_{D}=f \\ 
\left( \forall z\in \mathcal{U}\setminus \overline{D}\right) :\left \vert
\left( \frac{\partial }{\partial \overline{z}}\right) ^{N}F(z)\right \vert
\leq \mathcal{P}_{9}\exp \left( -\mathcal{P}_{10}\left( \left \vert z\right
\vert -1\right) ^{-k}\right)%
\end{array}%
\right.
\end{equation*}%
where $\mathcal{P}_{9},$\ $\mathcal{P}_{10}$\ $>0$\ are constants$.$ Let $%
R_{0}>0$ be such that the disk $D_{1+R_{0}}$ is contained in $\mathcal{U}$.
There exist (\cite{TOU}) a function $\Phi :%
\mathbb{C}
\longrightarrow \lbrack 0,1]$ of class $C^{\infty }$ on $%
\mathbb{C}
$ such that : 
\begin{equation*}
\left \{ 
\begin{array}{c}
\left( \forall z\in \overline{D}_{1+\frac{R_{0}}{3}}\right) :\Phi (z)=1 \\ 
\left( \forall z\in 
\mathbb{C}
\setminus \overline{D}_{1+\frac{2R_{0}}{3}}\right) :\Phi (z)=0%
\end{array}%
\right.
\end{equation*}%
We denote by $\widetilde{F}$ the function defined by :%
\begin{equation*}
\left \{ 
\begin{array}{c}
\left( \forall z\in \overline{D}_{1+\frac{2R_{0}}{3}}\right) :\widetilde{F}%
(z)=\Phi (z)F(z) \\ 
\left( \forall z\in 
\mathbb{C}
\setminus \overline{D}_{1+\frac{2R_{0}}{3}}\right) :\widetilde{F}(z)=0%
\end{array}%
\right.
\end{equation*}%
Then it is clear that the function $\widetilde{F}$ is an extension of the
function $f$ to $%
\mathbb{C}
,$ and that $\widetilde{F}$ is of class $C^{\infty }$ on $%
\mathbb{C}
$ with compact support contained in $\overline{D}_{1+\frac{2R_{0}}{3}}$ and
that the following estimate holds :%
\begin{equation*}
\left( \forall z\in \mathcal{%
\mathbb{C}
}\setminus \overline{D}\right) :\left \vert \left( \frac{\partial }{\partial 
\overline{z}}\right) ^{N}\widetilde{F}(z)\right \vert \leq \mathcal{P}%
_{11}\exp \left( -\mathcal{P}_{12}\left( \left \vert z\right \vert -1\right)
^{-k}\right)
\end{equation*}%
with some convenient constants $\mathcal{P}_{11},$\ $\mathcal{P}_{12}$\ $>0.$
Thanks to proposition $2.1.$, the following relations hold for every $z\in 
\mathbb{C}
:$%
\begin{eqnarray*}
\widetilde{F}\left( z\right) &=&\underset{%
\mathbb{C}
}{\diint }\frac{\left( \overline{z}-\overline{\zeta }\right) ^{N-1}}{\pi
\left( N-1\right) !\left( z-\zeta \right) }\left( \frac{\partial }{\partial 
\overline{z}}\right) ^{N}\widetilde{F}\left( \zeta \right) d\nu \left( \zeta
\right) \\
&=&\underset{D_{1+\frac{2R_{0}}{3}}\setminus \overline{D}}{\diint }\frac{%
\left( \overline{z}-\overline{\zeta }\right) ^{N-1}}{\pi \left( N-1\right)
!\left( z-\zeta \right) }\left( \frac{\partial }{\partial \overline{z}}%
\right) ^{N}\widetilde{F}\left( \zeta \right) d\nu \left( \zeta \right)
\end{eqnarray*}%
Let us denote by $\Psi $ the function : 
\begin{equation*}
\begin{array}{cccc}
\Psi : & 
\mathbb{C}
\times \left( D_{1+R_{0}}\setminus \overline{D}\right) & \rightarrow & 
\mathbb{C}
\\ 
& \left( z,\zeta \right) & \mapsto & \frac{\left( \overline{z}-\overline{%
\zeta }\right) ^{N-1}}{\pi \left( N-1\right) !\left( z-\zeta \right) }\left( 
\frac{\partial }{\partial \overline{z}}\right) ^{N}\widetilde{F}\left( \zeta
\right)%
\end{array}%
\end{equation*}%
Then we have the following estimates for each $\left( l,m\right) \in 
\mathbb{N}
^{2}:$%
\begin{eqnarray*}
&&\underset{\left( z,\zeta \right) \in D\times D_{1+\frac{2R_{0}}{3}},\text{ 
}\zeta \notin \overline{D}\cup \left \{ z\right \} }{\sup }\left \vert \frac{%
\partial ^{l+m}\Psi }{\partial z^{l}\partial \overline{z}^{m}}\left( z,\zeta
\right) \right \vert \\
&\leq &\underset{\left( z,\zeta \right) \in D\times D_{1+\frac{2R_{0}}{3}},%
\text{ }\zeta \notin \overline{D}\cup \left \{ z\right \} }{\sup }\max
\left( \left \vert z-\zeta \right \vert ,1\right) ^{N-1}\frac{\left \vert
\left( \frac{\partial }{\partial \overline{z}}\right) ^{N}\widetilde{F}%
\left( \zeta \right) \right \vert l!}{\pi \left \vert z-\zeta \right \vert
^{l+1}} \\
&\leq &\pi ^{-1}\mathcal{P}_{11}\left( 2\left( 1+\frac{2R_{0}}{3}\right)
\right) ^{N-1}\cdot \\
&&\cdot \underset{\left( z,\zeta \right) \in D\times D_{1+\frac{2R_{0}}{3}},%
\text{ }\zeta \notin \overline{D}\cup \left \{ z\right \} }{\sup }\frac{\exp
\left( -\mathcal{P}_{12}\left( \left \vert \zeta \right \vert -1\right)
^{-k}\right) l!}{\left \vert z-\zeta \right \vert ^{l+1}}
\end{eqnarray*}

But there exists, according to proposition $8.2.$\ in the appendix below$,$
a constant $\mathcal{P}_{13}>0$ such that : 
\begin{eqnarray*}
&&\underset{\left( z,\zeta \right) \in D\times D_{1+\frac{2R_{0}}{3}},\text{ 
}\zeta \notin \overline{D}\cup \left \{ z\right \} }{\sup }\frac{\exp \left(
-\mathcal{P}_{12}\left( \left \vert \zeta \right \vert -1\right)
^{-k}\right) l!}{\left \vert z-\zeta \right \vert ^{l+1}} \\
&\leq &\frac{1}{\pi ^{-1}\mathcal{P}_{11}\left( 2\left( 1+\frac{2R_{0}}{3}%
\right) \right) ^{N-1}}\mathcal{P}_{13}^{l+1}l^{\left( 1+\frac{1}{k}\right)
l}
\end{eqnarray*}

Consequently the following condition holds :%
\begin{equation*}
\left( \forall \left( l,m\right) \in 
\mathbb{N}
^{2}\right) :\underset{\left( z,\zeta \right) \in D\times D_{1+\frac{2R_{0}}{%
3}},\text{ }\zeta \notin \overline{D}\cup \left \{ z\right \} }{\sup }\left
\vert \frac{\partial ^{l+m}\Psi }{\partial z^{l}\partial \overline{z}^{m}}%
\left( z,\zeta \right) \right \vert \leq \mathcal{P}_{13}^{l+1}l^{\left( 1+%
\frac{1}{k}\right) l}
\end{equation*}%
It follows that, when applying the differential operator $\mathcal{L}_{lm}:=%
\frac{\partial ^{l+m}}{\partial z^{l}\partial \overline{z}^{m}}$ to $F,$ we
can interchange $\mathcal{L}_{lm}$ and the integral sign $\underset{D_{1+%
\frac{2R_{0}}{3}}\setminus \overline{D}}{\diint }$ to write for each $z\in
D: $%
\begin{eqnarray*}
\frac{\partial ^{l+m}f}{\partial z^{l}\partial \overline{z}^{m}}\left(
z\right) &=&\frac{\partial ^{l+m}\widetilde{F}}{\partial z^{l}\partial 
\overline{z}^{m}}\left( z\right) \\
&=&\underset{D_{1+\frac{2R_{0}}{3}}\setminus \overline{D}}{\diint }\frac{%
\partial ^{l+m}\Psi }{\partial z^{l}\partial \overline{z}^{m}}\left( z,\zeta
\right) d\nu (\zeta )
\end{eqnarray*}%
Hence we can write for each $z\in D$ : 
\begin{equation*}
\left \vert \frac{\partial ^{l+m}\widetilde{F}}{\partial z^{l}\partial 
\overline{z}^{m}}\left( z\right) \right \vert \leq \frac{4\pi R_{0}\left(
R_{0}+3\right) }{9}\mathcal{P}_{13}^{l+m+1}\left( l+m\right) ^{\left( 1+%
\frac{1}{k}\right) \left( l+m\right) }
\end{equation*}%
Consequently $f\in H_{N}^{k}\left( D\right) .$

The proof of the corollary is complete.
\end{proof}

\subsection{Degree of the best uniform $N$-polynomial approximation of
functions of the Gevrey class $H_{N}^{k}\left( D\right) $}

\begin{corollary}
\textit{For every }$k>0,$\textit{\ the set of the functions }%
\begin{equation*}
\begin{array}{cccc}
\Theta _{\alpha ,\beta }: & 
\mathbb{R}
^{+} & \rightarrow & 
\mathbb{R}
^{+} \\ 
& t & \mapsto & \alpha \exp \left( -\beta t^{\frac{k}{k+1}}\right)%
\end{array}%
\end{equation*}%
\textit{where }$\left( \alpha ,\beta \right) $ \textit{runs over} $%
\mathbb{R}
^{+\ast }\times 
\mathbb{R}
^{+\ast },$ \textit{is a degree of the best uniform }$N-$\textit{polynomial} 
\textit{approximation of the Gevrey class }$H_{N}^{k}\left( D\right) $%
\textit{.}
\end{corollary}

\begin{proof}
1. Let $f\in H_{N}^{k}\left( D\right) .$ According to the main result of
this paper there exists a sequence $(P_{n})_{n\in 
\mathbb{N}
^{\ast }}$\ of $N-$analytic polynomials\ such that :%
\begin{equation*}
\left \{ 
\begin{array}{c}
\left( \forall n\in 
\mathbb{N}
^{\ast }\right) :||P_{n}||_{\infty ,D_{k,R,n}}\leq C\delta ^{n} \\ 
\left( \forall z\in D\right) :\overset{+\infty }{\underset{n=1}{\sum }}%
P_{n}(z)=F(z) \\ 
\left( \forall n\in 
\mathbb{N}
^{\ast }\right) :d%
{{}^\circ}%
\left( P_{n}\right) \leq n^{\frac{k+1}{k}}%
\end{array}%
\right.
\end{equation*}%
where $C>0$ and $\delta \in ]0,1[$ are constants. Let us denote for each $%
n\in 
\mathbb{N}
^{\ast }$ by $Q_{n}$ the finite sum $Q_{n}:=\underset{j^{\frac{k+1}{k}}\leq n%
}{\sum }P_{j}.$ Then $Q_{n}\in \Pi _{N,n}$ and the following inequalities
hold :%
\begin{eqnarray*}
\mathcal{E}_{N,n}\left( f\right) &\leq &\left \Vert f-Q_{n}\right \Vert
_{\infty ,D} \\
&\leq &\underset{j^{\frac{k+1}{k}}>n}{\sum }\left \Vert P_{j}\right \Vert
_{\infty ,D} \\
&\leq &C\underset{j^{\frac{k+1}{k}}>n}{\sum }\delta ^{j} \\
&\leq &\frac{C}{1-\delta }\delta ^{n^{\frac{k}{k+1}}}
\end{eqnarray*}%
Consequently we have for each $n\in 
\mathbb{N}
:$%
\begin{equation*}
\mathcal{E}_{N,n}\left( f\right) \leq \mathcal{P}_{14}\exp \left( -\mathcal{P%
}_{1\emph{5}}n^{\frac{k}{k+1}}\right)
\end{equation*}%
where $\mathcal{P}_{14}:=\max \left( \frac{C}{1-\delta },\mathcal{E}%
_{N,0}\left( f\right) \right) >0$ and $\mathcal{P}_{1\emph{5}}:=\ln \left( 
\frac{1}{\delta }\right) >0.$

2. Conversly, let $f\in C\left( \overline{D}\right) $ which fullfiles the
following condition :%
\begin{equation*}
\left( \forall n\in 
\mathbb{N}
\right) :\mathcal{E}_{N,n}\left( f\right) \leq \mathcal{P}_{16}\exp \left( -%
\mathcal{P}_{17}n^{\frac{k}{k+1}}\right)
\end{equation*}%
where $\mathcal{P}_{16},\mathcal{P}_{17}>0$ are constants. Then there exists
for each $n\in 
\mathbb{N}
,$ a function $W_{n}\in \Pi _{N,n}$ such that :$\  \  \  \  \  \  \  \  \  \  \  \  \  \
\  \  \  \  \  \  \  \  \  \  \  \  \  \  \  \  \  \  \  \  \  \  \  \  \  \  \  \  \  \  \  \  \  \  \  \  \  \
\  \  \  \  \  \  \  \  \  \  \  \  \ $%
\begin{equation*}
\left \Vert f-W_{n}\right \Vert _{\infty ,D}\leq \mathcal{P}_{16}\exp \left(
-\mathcal{P}_{17}n^{\frac{k}{k+1}}\right)
\end{equation*}%
We denote for each $n\in 
\mathbb{N}
^{\ast },$ by $Y_{n}$ the finite sum$\underset{n^{\frac{k+1}{k}}\leq
j<\left( n+1\right) ^{\frac{k+1}{k}}}{\sum }\left( W_{j}-W_{j-1}\right) .$
Then $Y_{n}\in \Pi _{N}$ and we obtain, for every $z\in D_{k,\frac{\beta }{3}%
,n},$ by virtue of the proposition $3.3.$, the following estimates $:$%
\begin{eqnarray*}
&&\left \vert Y_{n}\left( z\right) \right \vert \leq \underset{n^{\frac{k+1}{%
k}}\leq j<\left( n+1\right) ^{\frac{k+1}{k}}}{\sum }\left \vert
W_{j}(z)-W_{j-1}(z)\right \vert \\
&\leq &\underset{n^{\frac{k+1}{k}}\leq j<\left( n+1\right) ^{\frac{k+1}{k}}}{%
\sum }\left( 2^{N+1}-1\right) \mathcal{J}_{N}\left( \frac{1}{2},1\right)
\cdot \\
&&\cdot \mathcal{P}_{16}\left( \exp \left( -\mathcal{P}_{17}j^{\frac{k}{k+1}%
}\right) +\exp \left( -\mathcal{P}_{17}\left( j-1\right) ^{\frac{k}{k+1}%
}\right) \right) \left \vert z\right \vert ^{j+N} \\
&\leq &\mathcal{P}_{16}\left( 2^{N+1}-1\right) \mathcal{J}_{N}\left( \frac{1%
}{2},1\right) \exp \left( -\mathcal{P}_{17}n\right) \cdot \\
&&\cdot \underset{n^{\frac{k+1}{k}}\leq j<\left( n+1\right) ^{\frac{k+1}{k}}}%
{\sum }\left( 1+\exp \left( \mathcal{P}_{17}\left( j^{\frac{k}{k+1}}-\left(
j-1\right) ^{\frac{k}{k+1}}\right) \right) \right) \cdot \\
&&\cdot \left( 1+\frac{\mathcal{P}_{17}}{3}n^{-\frac{1}{k}}\right) ^{j+N}
\end{eqnarray*}%
But, according to the proposition $8.3.$ in the appendix below, we have the
following asymptotic estimates :%
\begin{eqnarray*}
&&\exp \left( -\mathcal{P}_{17}n\right) \underset{n\leq j^{\frac{k}{k+1}}<n+1%
}{\sum }\left( 1+\exp \left( \mathcal{P}_{17}\left( j^{\frac{k}{k+1}}-\left(
j-1\right) ^{\frac{k}{k+1}}\right) \right) \right) \cdot \\
&&\cdot \left( 1+\frac{\mathcal{P}_{17}}{3}n^{-\frac{1}{k}}\right) ^{j+N} \\
&=&\underset{n\rightarrow +\infty }{O}\left( \left( e^{-\frac{\mathcal{P}%
_{17}}{4}}\right) ^{n}\right)
\end{eqnarray*}

Consequently there exist a constant $\mathcal{P}_{18}>0$ such that the
following estimate holds for each $n\in 
\mathbb{N}
:$ 
\begin{equation*}
\left \Vert Y_{n}\right \Vert _{\infty ,D_{k,\frac{\beta }{3},n}}\leq 
\mathcal{P}_{18}\left( e^{-\frac{\mathcal{P}_{17}}{4}}\right) ^{n}
\end{equation*}%
Since the function series $\sum Y_{n}$ is uniformly convergent on $D$ to the
function $f$ it follows, thanks to theorem $5.$ $1.$, that $f\in
H_{N}^{k}\left( D\right) .$

The proof of the corollary is then achieved.
\end{proof}

\section{Appendix}

\begin{proposition}
For each constant $\mathcal{B}>0,$ there exists a constant $\mathcal{D}%
_{1}>0 $ such that the following estimate holds $:$ 
\begin{eqnarray*}
\left( \forall n\in 
\mathbb{N}
\right) &:&e^{-\mathcal{B}n}\left( \left( n+1\right) ^{\frac{k+1}{k}}-n^{%
\frac{k+1}{k}}+1\right) \cdot \\
&&\cdot \left( 1+\frac{\mathcal{B}}{2}n^{-\frac{1}{k}}\right) ^{n^{\frac{k+1%
}{k}}} \\
&\leq &\mathcal{D}_{1}\left( e^{-\frac{\mathcal{B}}{4}}\right) ^{n}
\end{eqnarray*}
\end{proposition}

\begin{proof}
The following inequalities hold for each $n\in 
\mathbb{N}
:$%
\begin{eqnarray*}
&&e^{-\mathcal{B}n}\left( \left( n+1\right) ^{\frac{k+1}{k}}-n^{\frac{k+1}{k}%
}+1\right) \left( 1+\frac{\mathcal{B}}{2}n^{-\frac{1}{k}}\right) ^{n^{\frac{%
k+1}{k}}} \\
&\leq &e^{-\mathcal{B}n}e^{\frac{\mathcal{B}}{2}n}\left( \left( n+1\right) ^{%
\frac{k+1}{k}}-n^{\frac{k+1}{k}}+1\right) \\
&\leq &\left( \frac{k+1}{k}\right) e^{-\frac{\mathcal{B}}{2}n}\left( n^{%
\frac{1}{k}}+1\right) \\
&\leq &\left( \frac{k+1}{k}\right) e^{-\frac{\mathcal{B}}{4}n}\left( n^{%
\frac{1}{k}}+1\right) e^{-\frac{\mathcal{B}}{4}n}
\end{eqnarray*}%
Consequently there exist a constant $\mathcal{D}_{1}>0$ such that :%
\begin{eqnarray*}
\left( \forall n\in 
\mathbb{N}
\right) &:&e^{-\mathcal{B}n}\left( \left( n+1\right) ^{\frac{k+1}{k}}-n^{%
\frac{k+1}{k}}+1\right) \cdot \\
&&\cdot \left( 1+\frac{\mathcal{B}}{2}n^{-\frac{1}{k}}\right) ^{n^{\frac{k+1%
}{k}}} \\
&\leq &\mathcal{D}_{1}\left( e^{-\frac{\mathcal{B}}{4}}\right) ^{n}
\end{eqnarray*}

Thence we achieve the proof of the proposition.
\end{proof}

\begin{proposition}
For each constants $R,\mathcal{B}>0,$ there exists a constant $\mathcal{D}%
_{2}>0$ such that the following estimate holds $:$%
\begin{eqnarray*}
\left( \forall l\in 
\mathbb{N}
\right) &:&\underset{\left( z,\zeta \right) \in D\times D_{1+\frac{2R}{3}},%
\text{ }\zeta \notin \overline{D}\cup \left \{ z\right \} }{\sup }\left( 
\frac{\exp \left( -\mathcal{B}\left( \left \vert \zeta \right \vert
-1\right) ^{-k}\right) l!}{\left \vert z-\zeta \right \vert ^{l+1}}\right) \\
&\leq &\mathcal{D}_{2}^{l+1}l^{\left( 1+\frac{1}{k}\right) l}
\end{eqnarray*}
\end{proposition}

\begin{proof}
For each constants $R,\mathcal{B}>0,$ we have for every $l\in 
\mathbb{N}
:$%
\begin{eqnarray*}
&&\underset{\left( z,\zeta \right) \in D\times D_{1+\frac{2R}{3}},\text{ }%
\zeta \notin \overline{D}\cup \left \{ z\right \} }{\sup }\left( \frac{\exp
\left( -\mathcal{B}\left( \left \vert \zeta \right \vert -1\right)
^{-k}\right) l!}{\left \vert z-\zeta \right \vert ^{l+1}}\right) \\
&\leq &l!\underset{\left( z,\zeta \right) \in D\times D_{1+\frac{2R}{3}},%
\text{ }\zeta \notin \overline{D}\cup \left \{ z\right \} }{\sup }\left( 
\frac{\exp \left( -\mathcal{B}\left \vert z-\zeta \right \vert ^{-k}\right) 
}{\left \vert z-\zeta \right \vert ^{l+1}}\right) \\
&\leq &\mathcal{B}^{-\frac{l+1}{k}}l!\underset{t\geq 0}{\sup }\left( t^{%
\frac{l+1}{k}}e^{-t}\right)
\end{eqnarray*}%
But, according to the relation (\ref{boundsup}), we can write : 
\begin{eqnarray*}
&&\mathcal{B}^{-\frac{\left( l+1\right) }{k}}\underset{t\geq 0}{\sup }\left(
t^{\frac{l+1}{k}}e^{-t}\right) \\
&\leq &\left( \frac{\mathcal{B}}{k}\right) ^{\frac{1}{k}}\left( \left( \frac{%
k}{\mathcal{B}}\right) ^{\frac{1}{k}}\right) ^{l}l^{\frac{l}{k}}
\end{eqnarray*}%
Let us set $\mathcal{D}_{2}:=\max \left( \left( \frac{\mathcal{B}}{k}\right)
^{\frac{1}{k}},1\right) $. Then we have the following relation : 
\begin{eqnarray*}
\left( \forall l\in 
\mathbb{N}
\right) &:&\underset{\left( z,\zeta \right) \in D\times D_{1+\frac{2R}{3}},%
\text{ }\zeta \notin \overline{D}\cup \left \{ z\right \} }{\sup }\left( 
\frac{\exp \left( -\mathcal{B}\left( \left \vert \zeta \right \vert
-1\right) ^{-k}\right) l!}{\left \vert z-\zeta \right \vert ^{l+1}}\right) \\
&\leq &\mathcal{D}_{2}^{l+1}l^{\left( 1+\frac{1}{k}\right) l}
\end{eqnarray*}%
Thence we achieve the proof of the proposition.
\end{proof}

\begin{proposition}
For each constant $\mathcal{B}>0,$ the following asymptotic estimate holds $%
: $%
\begin{eqnarray*}
&&e^{-\mathcal{B}n}\underset{n^{\frac{k+1}{k}}\leq j<\left( n+1\right) ^{%
\frac{k+1}{k}}}{\sum }\left( 1+\exp \left( \mathcal{B}\left( j^{\frac{k}{k+1}%
}-\left( j-1\right) ^{\frac{k}{k+1}}\right) \right) \right) \cdot \\
&&\cdot \left( 1+\frac{\mathcal{B}}{3}n^{-\frac{1}{k}}\right) ^{j+N} \\
&=&\underset{n\rightarrow +\infty }{O}\left( \left( e^{-\frac{\mathcal{B}}{4}%
}\right) ^{n}\right)
\end{eqnarray*}
\end{proposition}

\begin{proof}
The following asymptotic relations hold for each constant $\mathcal{B}>0:$ 
\begin{eqnarray*}
&&e^{-\mathcal{B}n}\underset{n^{\frac{k+1}{k}}\leq j<\left( n+1\right) ^{%
\frac{k+1}{k}}}{\sum }\left( 1+\exp \left( \mathcal{B}\left( j^{\frac{k}{k+1}%
}-\left( j-1\right) ^{\frac{k}{k+1}}\right) \right) \right) \left( 1+\frac{%
\mathcal{B}}{3}n^{-\frac{1}{k}}\right) ^{j+N} \\
&=&\underset{n\rightarrow +\infty }{O}\left( e^{-\mathcal{B}n}\exp \left(
\left( 1+\underset{n\rightarrow +\infty }{o}\left( 1\right) \right) \frac{%
\mathcal{B}}{3}n\right) \left( \underset{n^{\frac{k+1}{k}}\leq j<\left(
n+1\right) ^{\frac{k+1}{k}}}{\sum }1\right) \right) \\
&=&\underset{n\rightarrow +\infty }{O}\left( e^{\frac{2\mathcal{B}}{3}n}e^{-%
\mathcal{B}n}\left( \left( n+1\right) ^{\frac{k+1}{k}}-n^{\frac{k+1}{k}%
}\right) \right) \\
&=&\underset{n\rightarrow +\infty }{O}\left( n^{\frac{1}{k}}e^{-\frac{%
\mathcal{B}}{3}n}\right) \\
&=&\underset{n\rightarrow +\infty }{O}\left( \left( e^{-\frac{\mathcal{B}}{4}%
}\right) ^{n}\right)
\end{eqnarray*}%
Thence we achieve the proof of the proposition.
\end{proof}

\bigskip

\textbf{Acnowlegments :}\emph{\ We would like to express our deep gratitude
to professor Said Asserda and to professor Karim Kellay for their precious
bibliographic help who allowed us to discover the wonderful world of
polyanalytic functions. Our thanks to the referee for her/his remarks and
suggestions who have improved this paper.\bigskip }

\end{document}